\documentclass[12pt]{amsart}
\usepackage{amscd,amsmath,amssymb,enumerate,amsfonts,mathrsfs,txfonts}
\usepackage{comment}
\usepackage{hyperref}
\usepackage{xcolor}
\usepackage{tikz,pgfplots,pgf}
\usepackage[all]{xy}

\tikzset{node distance=3cm, auto}
\usetikzlibrary{calc} 
\usetikzlibrary{trees}

\usepackage{vmargin}
\setpapersize{A4}
\setmargins{2.5cm}      
{1.5cm}                        
{16.5cm}                      
{23.42cm}                   
{10pt}                          
{1cm}                           
{0pt}                          
{2cm}
\newtheorem{theorem}{Theorem}[section]
\newtheorem{lemma}[theorem]{Lemma}
\newtheorem{proposition}[theorem]{Proposition}
\newtheorem{corollary}[theorem]{Corollary}
\theoremstyle{definition}
\newtheorem{definition}[theorem]{Definition}
\newtheorem{example}[theorem]{Example}
\newtheorem{examples}[theorem]{Examples}

\newtheorem{remark}[theorem]{Remark}
\numberwithin{equation}{section}

\def\R{\mathbb{R}}
\def\C{\mathbb{C}}
\def\T{\mathbb{T}}
\def\D{\mathbb{D}}

\def\Re{\mathrm{Re}}
\def\co{\mathrm{co}}
\def\abco{\mathrm{abco}}

\def\diam{\mathrm{diam}}

\def\lin{\mathrm{lin}}

\def\At{\mathrm{At}}
\def\Ext{\mathrm{Ext}}
\def\NA{\mathrm{NA}}

\makeatletter
\@namedef{subjclassname@2020}{\textup{2020} Mathematics Subject Classification}
\makeatother

\begin{document}

\title[The Bishop--Phelps--Bollob\'as property for weighted holomorphic mappings]{The Bishop--Phelps--Bollob\'as property \\ for weighted holomorphic mappings}

\author[A. Jim{\'e}nez-Vargas]{A. Jim{\'e}nez-Vargas}
\address[A. Jim{\'e}nez-Vargas]{Departamento de Matem{\'a}ticas, Universidad de Almer{\'i}a, 04120, Almer{\'i}a, Spain}
\email{ajimenez@ual.es}

\author[M. I. Ram\'{i}rez]{M. I. Ram{\'\i}rez}
\address[M. I. Ram\'{i}rez]{Departamento de Matem{\'a}ticas, Universidad de Almer{\'i}a, 04120, Almer{\'i}a, Spain}
\email{mramirez@ual.es}

\author[Mois\'es Villegas-Vallecillos]{Mois\'{e}s Villegas-Vallecillos}
\address[Mois\'es Villegas-Vallecillos]{Departamento de Matem\'{a}ticas, Universidad de C\'{a}diz, 11510 Puerto Real, Spain}
\email{moises.villegas@uca.es}

\date{\today}

%46B20 Geometry and structure of normed linear spaces
%46E50 Spaces of differentiable or holomorphic functions on infinite-dimensional spaces
%46T25 Holomorphic maps in nonlinear functional analysis

\thanks{Corresponding author: A. Jim\'enez-Vargas (ajimenez@ual.es)}

\begin{abstract}
Given an open subset $U$ of a complex Banach space $E$, a weight $v$ on $U$ and a complex Banach space $F$, let $H^\infty_v(U,F)$ denote the Banach space of all weighted holomorphic mappings from $U$ into $F$, endowed with the weighted supremum norm. We introduce and study a version of the Bishop--Phelps--Bollob\'as property for $H^\infty_v(U,F)$ ($WH^\infty$-BPB property, for short). A result of Lindenstrauss type with sufficient conditions for $H^\infty_v(U,F)$ to have the $WH^\infty$-BPB property for every space $F$ is stated. This is the case of $H^\infty_{v_p}(\D,F)$ with $p\geq 1$, where $v_p$ is the standard polynomial weight on $\D$. The study of the relations of the $WH^\infty$-BPB property for the complex and vector-valued cases is also addressed as well as the extension of the cited property for mappings $f\in H^\infty_v(U,F)$ such that $vf$ has a relatively compact range in $F$.
\end{abstract}
\maketitle

%%%%%%%%%%%%%%%%%%%%%%%%%%%%%%%%%%%%%%%%%%%%%%%%%%%%%%%%%%%%%%%%%%%%%%%%%%%%%%%%%%%%%%%%%%%%%%%%%%%%%%%%%%%%%%%%%%%%%%%%%%%%%%%%%%%%%%%%%%%%%%%%%%%%%%%%%%%%%%%%%%%%%%%%%%%%%%%%%%%%%%%%%%%%%%%%%%%%%%%%%%%%%

\section*{Introduction and preliminaries}\label{s0}

Bishop--Phelps Theorem \cite{BisPhe-61} states that every continuous linear functional on a Banach space can be approximated by a norm attaining linear functional. A strengthening of Bishop--Phelps Theorem, known as Bishop--Phelps--Bollob\'as Theorem \cite{Bol-70}, assures that, in addition, a point where the approximated functional almost attains its norm can be approximated by a point at which attains its norm. Acosta, Aron, Garc\'ia, and Maestre \cite{AcoAroGarMae-08} initiated the study of the Bishop--Phelps--Bollob\'as property for bounded linear operators between Banach spaces. 
%Lindenstrauss \cite{Lin-63} and Acosta, Aron, Garc\'ia and Maestre \cite{AcoAroGarMae-08} initiated the study of the Bishop--Phelps property and the Bishop--Phelps--Bollob\'as property for bounded linear operators between Banach spaces, respectively. 

Let $E$ and $F$ be Banach spaces and let $L(E,F)$ be the Banach space of all bounded linear operators from $E$ into $F$, endowed with the operator canonical norm. In particular, $E^*$ stands for the space $L(E,\mathbb{K})$. As usual, $B_E$ and $S_E$ denote the closed unit ball and the unit sphere of $E$, respectively. 

Let us recall (see \cite{AcoAroGarMae-08,Aco-19}) that the pair $(E,F)$ \emph{has the Bishop--Phelps--Bollob\'as property} if for any $0<\varepsilon<1$, there exists $0<\eta(\varepsilon)<\varepsilon$ such that for every operator $T\in S_{L(E,F)}$ and every point $x\in S_E$ such that $\left\|T(x)\right\|>1-\eta(\varepsilon)$, there exist $T_0\in S_{L(E,F)}$ and $x_0\in S_E$ satisfying $\left\|T_0(x_0)\right\|=1$, $\left\|T-T_0\right\|<\varepsilon$ and $\left\|x-x_0\right\|<\varepsilon$. 

A vast research on this topic has been carried out over time. The survey \cite{DanGarMaeRol-22} by Dantas, Garc\'ia, Maestre, and Rold\'an provides an overview from 2008 to 2021 about the Bishop--Phelps--Bollob\'as property in several directions: for operators, some classes of operators and multilinear forms (see also the survey \cite{Aco-19} by Acosta for these three lines), for homogeneous polynomials, for Lipschitz mappings and for holomorphic functions. 

In \cite[p. 539]{DanGarMaeRol-22}, it is stated that little is known about the Bishop--Phelps--Bollob\'as property for holomorphic mappings and it is suggested that its study deserves special attention. In this direction, non-linear versions of Bishop--Phelps--Bollob\'as Theorem were established for some classes of holomorphic functions by Dantas, Garc\'ia, S.K. Kim, U.Y. Kim, Lee and Maestre \cite{DanGarKimKimLeeMae-19} and by S.K. Kim and Lee \cite{KimLee-19}, and for operators between spaces of bounded holomorphic functions by Bala, Dhara, Sarkar and Sensarma \cite{BaDhaSarSen-23}. 

Motivated also by some results obtained in \cite{Jim-23} about the density of norm attaining weighted holomorphic mappings on open subsets of $\C^n$, our aim in this paper is to address the Bishop--Phelps--Bollob\'as property for weighted holomorphic mappings under a different approach. 

Let $E$ and $F$ be complex Banach spaces and let $U$ be an open subset of $E$. Let $H(U,F)$ denote the space of all holomorphic mappings from $U$ into $F$. A \textit{weight} $v$ on $U$ is a (strictly) positive continuous function on $U$.

The \textit{weighted space of holomorphic mappings} $H^\infty_v(U,F)$ is the Banach space of all mappings $f\in H(U,F)$ such that   
$$
\left\|f\right\|_v:=\sup\left\{v(z)\left\|f(z)\right\|\colon z\in U\right\}<\infty ,
$$ 
equipped with the \textit{weighted supremum norm} $\left\|\cdot\right\|_v$. Moreover, $H^\infty_{vK}(U,F)$ stands for the space of all mappings $f\in H^\infty_v(U,F)$ such that $vf$ has a relatively compact range in $F$. It is usual to write $H^\infty_v(U)$ instead of $H^\infty_v(U,\C)$. 

%The \textit{little weighted space of holomorphic mappings} $H^0_{v}(U,F)$ is the norm-closed linear subspace of $H^\infty_v(U,F)$ consisting of all those mappings $f$ so that $vf$ vanishes at infinity on $U$, that is, for every $\varepsilon>0$ there is a compact set $K\subseteq U$ such that $v(z)\left\|f(z)\right\|<\varepsilon$ for all $z\in U\setminus K$. It is usual to write $H^\infty_v(U)$ and $H^0_v(U)$ instead of $H^\infty_v(U,\C)$ and $H^0_v(U,\C)$, respectively. 

These spaces appear in the study of growth conditions of holomorphic functions. Some of the most important references about them are the works by Bierstedt, Bonet and Galbis \cite{BieBonGal-93} and Bierstedt and Summers \cite{BieSum-93}. For complete and recent information on such spaces, we refer the reader to the survey \cite{Bon-22} by Bonet and the references therein. 

Our approach requires the linearisation of weighted holomorphic mappings. This technique can be consulted in the works by Bonet, Domanski and Lindstr\"om \cite{BonDomLin-01} and Gupta and Baweja \cite{GupBaw-16}. See also the papers \cite{Muj-91} by Mujica for the case of bounded holomorphic mappings and \cite{BonFri-02} by Bonet and Friz for more general weighted spaces of holomorphic mappings.

The linearisation of functions has been employed to address similar problems in the setting of Lipschitz mappings by Cascales, Chiclana, Garc\'ia-Lirola, Mart\'in and Rueda-Zoca \cite{CasChiGarMarRue-19} and Chiclana and Mart\'in \cite{ChiMa-19}, and in the environment of holomophic mappings by Carando and Mazzitelli \cite{CarMaz-15} and the first author \cite{Jim-23}.

Since $B_{H^\infty_v(U)}$ is compact for the compact open topology $\tau_0$ by Ascoli's Theorem, it follows by Ng's Theorem \cite{Ng-71} that $H^\infty_v(U)$ is a dual Banach space and its predual, denoted $G^\infty_v(U)$, is defined as the space of all linear functionals on $H^\infty_v(U)$ whose restrictions to $B_{ H^\infty_v(U)}$ are $\tau_0$-continuous.%, endowed with the topology of uniform convergence on the set $B_{H^\infty_v(U)}$.  

For each $z\in U$, the \emph{evaluation functional} $\delta_z\colon H^\infty_v(U)\to\mathbb{C}$, defined by $\delta_z(f)=f(z)$ for $f\in H^\infty_v(U)$, is in $G^\infty_v(U)$. By \emph{an atom of $G^\infty_v(U)$} we mean an element of $G^\infty_v(U)$ of the form $v(z)\delta_z$ for $z\in U$. The set of all atoms in $G^\infty_v(U)$ will be denoted here by $\At_{G^\infty_v(U)}$. 

Given a Banach space $E$, a subset $N\subseteq B_{E^*}$ is said to be \textit{norming for $E$} if  
$$
\left\|x\right\|=\sup\left\{\left|x^*(x)\right|\colon x^*\in N\right\}\qquad (x\in E).
$$
Notice that $\At_{G^\infty_v(U)}$ is norming for $H^\infty_v(U)$ since
$$
\left\|f\right\|_v=\sup\left\{v(z)\left|f(z)\right|\colon z\in U\right\}=\sup\left\{\left|(v(z)\delta_z)(f)\right|\colon z\in U\right\}
$$
for every $f\in H^\infty_v(U)$.

%We will need the following linearisation theorem for $ H^\infty_v(U)$ and some of its applications. First, 
We now fix some notations. Given Banach spaces $E$ and $F$, we denote by $L((E,\mathcal{T}_E);(F,\mathcal{T}_F))$ the space of all continuous linear operators from $(E,\mathcal{T}_E)$ into $(F,\mathcal{T}_F)$, where $\mathcal{T}_E$ and $\mathcal{T}_F$ are topologies on $E$ and $F$, respectively. We will not write $\mathcal{T}_E$ whenever it is the norm topology of $E$. Hence $L(E,F)$ is the Banach space of all bounded linear operators from $E$ into $F$ with the canonical norm of operators. $K(E,F)$ is the norm-closed subspace of $L(E,F)$ consisting of all compact operators. As usual, $w^*$, $w$ and $bw^*$ denote the weak* topology, the weak topology, and the bounded weak* topology, respectively. $\Ext(B_E)$ represents the set of extreme points of $B_E$. $\D$ and $\T$ stand for the open unit ball and the unit sphere of $\mathbb{C}$, respectively. Given a set $A\subseteq E$, $\overline{\lin}(A)$, $\overline{\co}(A)$ and $\overline{\abco}(A)$ denote the norm-closed linear hull, the norm-closed convex hull and the norm-closed absolutely convex hull of $A$ in $E$. %It is easy to prove that $\overline{\abco}(A)=\overline{\co}(\overline{\D}A)=\overline{\co}(\T A)$.

\begin{theorem}\label{main-theo}\cite{BonDomLin-01,BonFri-02,GupBaw-16,Muj-91} 
Let $U$ be an open subset of a complex Banach space $E$ and $v$ be a weight on $U$.
\begin{enumerate}
\item $G^\infty_v(U)$ is a Banach space with the norm induced by $H^\infty_v(U)^*$ (in fact, a closed subspace of $H^\infty_v(U)^*$), and the evaluation mapping $J_v\colon H^\infty_v(U)\to G^\infty_v(U)^*$, given by $J_v(f)(\phi)=\phi(f)$ for $\phi\in G^\infty_v(U)$ and $f\in H^\infty_v(U)$, is an isometric isomorphism.
%\item The restriction mapping $R_v\colon G^\infty_v(U)\to H^0_v(U)^*$, given by $R_v(\phi)=\phi|_{ H^0_v(U)}$ for $\phi\in G^\infty_v(U)$, is an isometric isomorphism if and only if $B_{H^0_v(U)}$ is $\tau_0$-dense in $B_{H^\infty_v(U)}$.% for the compact open topology.
\item The mapping $\Delta_v\colon U\to G^\infty_v(U)$ defined by $\Delta_v(z)=\delta_z$ for $z\in U$, is in $H^\infty_v(U,G^\infty_v(U))$ with $\left\|\Delta_v\right\|_v\leq 1$. 
\item $B_{G^\infty_v(U)}$ coincides with $\overline{\abco}(\At_{G^\infty_v(U)})\subseteq H^\infty_v(U)^*$.
\item $G^\infty_v(U)$ coincides with $\overline{\lin}(\At_{G^\infty_v(U)})\subseteq H^\infty_v(U)^*$.
\item For every complex Banach space $F$ and every mapping $f\in H^\infty_v(U,F)$, there exists a unique operator $T_f\in L(G^\infty_v(U),F)$ such that $T_f\circ\Delta_v=f$. Furthermore, $||T_f||=\left\|f\right\|_v$.
\item The mapping $f\mapsto T_f$ is an isometric isomorphism from $H^\infty_v(U,F)$ onto $L(G^\infty_v(U),F)$ (resp., from $H^\infty_{vK}(U,F)$ onto $K(G^\infty_v(U),F))$. $\hfill\qed$
\end{enumerate}
\end{theorem}

In the case $v=1_U$ where $1_U(z)=1$ for all $z\in U$, we will simply write $H^\infty(U,F)$ (the Banach space of all bounded holomorphic mappings from $U$ into $F$, under the supremum norm) instead of $H^\infty_v(U,F)$, $H^\infty(U)$ in place of $H^\infty(U,\C)$ and, following Mujica's notation in \cite{Muj-91}, $G^\infty(U)$ instead of $G^\infty_v(U)$.

Let us recall that an operator $T\in L(E,F)$ is said \textit{to attain its norm} at a point $x\in S_E$ if $\left\|T(x)\right\|=\left\|T\right\|$. Usually, $\NA(E,F)$ denotes the set of all operators in $L(E,F)$ that attain their norms and, in particular, $\NA(E)$ stands for $\NA(E,\mathbb{K})$. 

We can introduce the following version of this concept for weighted holomorphic mappings. 

\begin{definition}
Let $U$ be an open subset of a complex Banach space $E$, let $v$ be a weight on $U$, let $F$ be a complex Banach space, and $f\in H^\infty_v(U,F)$. 
\begin{enumerate}
\item We say that $f$ \textit{attains its weighted supremum norm} if there exists a point $z\in U$ such that $v(z)\left\|f(z)\right\|=\left\|f\right\|_v$. We denote by $ H^\infty_{v\rm{NA}}(U,F)$ the set of all mappings $f\in H^\infty_v(U,F)$ attaining their weighted supremum norms. In particular, we write $ H^\infty_{v\NA}(U)$ instead of $ H^\infty_{v\NA}(U,\C)$.
\item We say that $f$ \textit{attains its weighted supremum norm on $ G^\infty_v(U)$} if its linearisation $T_f\in L( G^\infty_v(U),F)$ attains its operator canonical norm. The set of all mappings $f\in H^\infty_v(U,F)$ that attain their weighted supremum norms on $ G^\infty_v(U)$ is denoted by $ H^\infty_{v\NA}(G^\infty_v(U),F)$. In addition, we write $ H^\infty_{v\NA}(G^\infty(U))$ in place of $ H^\infty_{v\NA}(G^\infty(U),\C)$.
\end{enumerate}
\end{definition}

\begin{example}
It is clear that if $f$ attains its weighted supremum norm (at $z$), then $f$ attains its weighted supremum norm on $ G^\infty_v(U)$ (at $v(z)\delta_z$). The converse does not hold: for example, the function identity on $\D$ does not attain its supremum norm on $\D$, but it does on $ G^\infty(\D)$. %Indeed, if $f(z)=z$ for all $z\in\D$, we have $|f(z)|=|z|<1$ for all $z\in\D$ and thus $\left\|f\right\|_\infty\leq 1$, moreover $1-1/n\in\D$ and $1-1/n=|f(1-1/n)|\leq \left\|f\right\|_\infty\leq 1$ for all $n\in\mathbb{N}$ and hence $\left\|f\right\|_\infty=1$, however there does not exist $z_0\in\D$ such that $|z_0|=|f(z_0)|=1$. On the other hand, take any $z\in\D$ with $z\neq 0$ and denote $w=z/|z|$, we have that $\delta_w\in B_{G^\infty(\D)}$ and $1=|f(w)|=|T_f(\delta_w)|\leq \left\|T_f\right\|=\left\|f\right\|_\infty=1$. 
\end{example}

In view of the definition of the weighted supremum norm, a possible formulation of the Bishop--Phelps--Bollob\'as property in the setting of weighted holomorphic mappings could be the following.

\begin{definition}
Let $U$ be an open subset of a complex Banach space $E$, $v$ be a weight on $U$, and $F$ be a complex Banach space. We say that $ H^\infty_v(U,F)$ \emph{has the weighted holomorphic Bishop--Phelps--Bollob\'as property ($WH^\infty$-BPB property, for short)} if given $0<\varepsilon<1$, there is $0<\eta(\varepsilon)<\varepsilon$ such that for every $f\in S_{H^\infty_v(U,F)}$, every $\lambda\in\T$ and every $z\in U$ such that $v(z)\left\|f(z)\right\|>1-\eta(\varepsilon)$, there exist $f_0\in S_{H^\infty_v(U,F)}$, $\lambda_0\in\T$ and $z_0\in U$ such that $v(z_0)\left\|f_0(z_0)\right\|=1$, $\left\|f-f_0\right\|_v<\varepsilon$ and $\left\|\lambda v(z)\delta_z-\lambda_0 v(z_0)\delta_{z_0}\right\|<\varepsilon$. In this case, it is said that $H^\infty_v(U,F)$ \emph{has the $WH^\infty$-BPB property with function $\varepsilon\mapsto\eta(\varepsilon)$}.

If the preceding definition holds for a linear subspace $A^\infty_v(U,F)\subseteq H^\infty_v(U,F)$ (that is, $f$ and $f_0$ belong to $S_{A^\infty_v(U,F)}$), we say that $A^\infty_v(U,F)$ \emph{has the $WH^\infty$-BPB property}.
\end{definition}

%\begin{remark}
%The Bishop--Phelps--Bollob\'as property has been introduced in other settings In contrast with the 
%\end{remark}

%If $v$ is a weight on $U$ such that $v(z)\leq 1$ for all $z\in U$, then the constant function $1_U$ is in $H^\infty_v(U)$ with $\left\|1_U\right\|_v\leq 1$, therefore $|v(z)-v(z_0)|\leq ||\lambda v(z)\delta_z-\lambda_0 v(z_0)\delta_{z_0}||$ for any $z,z_0\in U$ and $\lambda,\lambda_0\in\T$, and thus the quantity $||\lambda v(z)\delta_z-\lambda_0v(z_0)\delta_{z_0}||$ involved in the preceding definition provides an estimation of the distance between the weights of $v$ at $z$ and $z_0$.

%Here are examples of weights $v$ on $\D$ satisfying that $v(z)\leq 1$ for all $z\in\D$. For any $p>0$:
%\begin{enumerate}
%	\item The \emph{polynomial weights} $v_p(z)=(1-|z|^2)^p$.
%	\item The \emph{exponential weights} $v_p(z)=\mathrm{exp}(-(1-|z|^2)^{-p})$.
	%\item The \emph{logarithmic weights} $v_p(z)=(1-\log(1-|z|^2))^{-p}$.
%\end{enumerate}

\begin{examples}
Let $\Omega\subseteq\C$ be a simply connected open set. If $\Omega=\C$, then $ H^\infty(\Omega)=\C$ by Liouville's Theorem and thus $ H^\infty(\Omega)$ has the $WH^\infty$-BPB property. %Since every $f\in H^\infty(\C)$ with $\left\|f\right\|_\infty=1$ is a constant function with modulus 1, we can take $f_0=f$ and $z_0$ any complex number.

If $\Omega\varsubsetneq\C$, we can suppose that $\Omega=\D$ by the Riemann Mapping Theorem. By the maximum modulus principle, we have $H_{\NA}^\infty(\Omega)=\C$, hence $H_{\NA}^\infty(\Omega)$ is not norm dense in $ H^\infty(\Omega)$ and, therefore, $ H^\infty(\Omega)$ fails the $WH^\infty$-BPB property.
\end{examples}

%Clearly, if $H^\infty_v(U,F)$ has the $WH^\infty$-BPB property, then $H^\infty_{v\NA}(U,F)$ is norm dense in $ H^\infty_v(U,F)$. 

%\begin{example}
%OJO Find an example such that $H^\infty_{v\NA}(U,F)$ is norm dense in $H^\infty_v(U,F)$ but $H^\infty_v(U,F)$ fails the $WH^\infty$-BPB property.
%\end{example}

%\begin{lemma}\label{lem-1}
%Let $U$ be an open subset of $\C^n$, let $v$ be a weight on $U$ and let $F$ be a complex Banach space. The following assertions are equivalent:
%\begin{enumerate}
%	\item $H^\infty_v(U,F)$ has the $WH^\infty$-BPB property.
%	\item Given $\varepsilon>0$ there is $\eta(\varepsilon)>0$ such that for every $S_f\in L( G^\infty_v(U),F)$ with $\left\|S_f\right\|=1$ and every $z\in U$ such that $\left\|S_f(v(z)\delta_z)\right\|>1-\eta(\varepsilon)$, there exist $S_g\in L( G^\infty_v(U),F)$ and $x\in U$ such that $\left\|S_g(v(x)\delta_x)\right\|=\left\|g\right\|_v=1$, $\left\|S_f-S_g\right\|<\varepsilon$ and $\left\|v(z)\delta_z-v(x)\delta_x\right\|<\varepsilon$. 
%\end{enumerate}
%\end {lemma}

%Let us recall that a subset $U$ of a complex Banach space $E$ is \textit{balanced} if $\lambda z\in U$ for all $z\in U$ and $\lambda\in\overline{\mathbb{D}}$. 

We now present the content of the paper. The main result of this paper assures in Section \ref{s1} that if $U$ is an open subset of a complex Banach space $E$ and $v$ is a weight on $U$ such that $\T\At_{G^\infty_v(U)}$ is a norm-closed set of uniformly strongly exposed points of $B_{G^\infty_v(U)}$, then $H^\infty_v(U,F)$ has the $WH^\infty$-BPB property for every complex Banach space $F$. This is the case of $H^\infty_{v_p}(\D,F)$ with $p\geq 1$, where $v_p$ is the \emph{polynomial weight} on $\D$ defined by $v_p(z)=(1-|z|^2)^p$ for all $z\in\D$. Our approach requires a foray into the study of the extremal structure of the unit closed ball of the space $G^\infty_{v_p}(U)$. This rich structure has been studied by Boyd and Rueda \cite{BoyRue-06} to develop the geometric theory of the space $H^\infty_v(U)$.

In Section \ref{s2}, we show that the $WH^\infty$-BPB property for mappings $f\in H^\infty_v(U,F)$ implies the $WH^\infty$-BPB property for functions $f\in H^\infty_v(U)$, and that the converse implication holds whenever the space $F$ enjoys the Lindenstrauss' property $\beta$. 

Finally, we devote Section \ref{s3} to state analogous results for the space $H^\infty_{vK}(U,F)$, and we also show that $H^\infty_{vK}(U,F)$ has the $WH^\infty$-BPB property whenever $H^\infty_{v}(U)$ enjoys this property and $F$ is a predual of a complex $L_1(\mu)$-space.

%%%%%%%%%%%%%%%%%%%%%%%%%%%%%%%%%%%%%%%%%%%%%%%%%%%%%%%%%%%%%%%%%%%%%%%%%%%%%%%%%%%%%%%%%%%%%%%%%%%%%%%%%%%%%%%%%%%%%%%%%%%%%%%%%%%%%%%%%%%%%%%%%%%%%%%%%%%%%%%%%%%%%%%%%%%%%%%%%%%%%%%%%%%%%%%%%%%%%%%%%%%%%

\section{Weighted spaces of holomorphic mappings with the $WH^\infty$-BPB property}\label{s1}

Following to Lindenstrauss \cite{Lin-63}, a Banach space $E$ is said to \textit{have the property A} if $\NA(E,F)$ is norm dense in $ L(E,F)$ for every Banach space $F$. To give a sufficient condition for a Banach space $E$ to enjoy the property A, the notion of a set of uniformly strongly exposed points of $B_{E}$ was considered in \cite{Lin-63}. 

Let $E$ be a complex Banach space. A point $x\in B_E$ is said to be an \textit{exposed point of $B_E$} if there exists a functional $f\in S_{E^*}$ such that $\mathrm{Re}(f(x))=1$ and $\mathrm{Re}(f(y))<1$ for all $y\in B_E$ with $y\neq x$. A point $x\in B_E$ is a \textit{strongly exposed point of $B_E$} if there exists a functional $f\in S_{E^*}$ such that $f(x)=1$ and satisfies the following condition: for every $0<\varepsilon<1$, there exists a $0<\delta<1$ such that if $y\in B_E$ and $\mathrm{Re}(f(y))>1-\delta$, then $\left\|y-x\right\|<\varepsilon$. 

%It is clear that every strongly exposed point $x\in B_E$ is exposed. Indeed, there is a $f\in S_{E^*}$ such that $f(x)=1$ and thus $\mathrm{Re}(f(x))=1$. Moreover, if $y\in B_E$ with $y\neq x$, we can find $\varepsilon$ such that $0<\varepsilon<\left\|x-y\right\|$, hence there exists a $\delta>0$ such that if $z\in B_E$ and $\mathrm{Re}(f(z))>1-\delta$, then $\left\|z-x\right\|<\varepsilon$. It follows that $\mathrm{Re}(f(y))\leq 1-\delta$ since otherwise we would have $\mathrm{Re}(f(y))>1-\delta$ and therefore $\left\|y-x\right\|<\varepsilon$, a contradiction.

%Note that every exposed point of $B_E$ is in $S_E$.  

Given $f\in S_{E^*}$ and $0<\delta<1$, the \emph{slice of $B_E$ associated to $f$ and $\delta$} is the set
$$
S(B_E,f,\delta)=\left\{x\in B_E\colon \mathrm{Re}(f(x))>1-\delta\right\}.
$$
A subset $S\subseteq S_E$ is said to be a \emph{set of uniformly strongly exposed points of $B_E$} if there exists a set of functionals $\{f_x\colon x\in S\}\subseteq S_{E^*}$ with $f_x(x)=1$ for every $x\in S$ such that, given $0<\varepsilon<1$, there is $0<\delta<1$ satisfying that $\diam(S(B_E,f_x,\delta))<\varepsilon$ for all $x\in S$. 
%that is, 
%$$
%\sup\left\{\diam(S(B_E,f_x,\delta))\colon x\in S\right\}<\varepsilon.%,
%$$
%Equivalently, if for every $\varepsilon>0$ there is $\delta'>0$ such that whenever $z\in B_E$ satisfies $\mathrm{Re}(f_x(z))>1-\delta'$ for some $x\in S$, then $\left\|x-z\right\|<\varepsilon$ (that is, all elements of $S$ are strongly exposed points with the same relation $\varepsilon$-$\delta$). 
In this case, it is said that \emph{$B_E$ is uniformly strongly exposed by the family of functionals $\{f_x\colon x\in S\}$}. 

%Note that a set $S\subseteq S_E$ of uniformly strongly exposed points of $B_E$ is a set of strongly exposed points of $B_E$. Indeed, $f_x(x)=1$ for every $x\in S$. Now, let $\varepsilon>0$. By hypothesis, there exists $0<\delta<1$ so that $\diam(S(B_E,f_x,\delta))<\varepsilon$ for all $x\in S$. If $y\in B_E$ and $\mathrm{Re}(f_x(y))>1-\delta$, then $y\in S(B_E,f_x,\delta)$ and thus $\left\|y-x\right\|\leq\diam(S(B_E,f_x,\delta))<\varepsilon$. This proves that the points of $S$ are strongly exposed. 

In \cite[Proposition 1]{Lin-63}, %(see also \cite[Remark 9]{Aco-91}), 
Lindenstrauss proved that if $E$ is a Banach space containing a set of uniformly strongly exposed points $S\subseteq S_E$ such that $B_E=\overline{\co}(S)$, then $E$ has the property A. In fact, a reading of its proof shows that for every Banach space $F$, the set
$$
\left\{T\in L(E,F)\colon \exists x\in\overline{S} \, | \, \left\|T(x)\right\|=\left\|T\right\|\right\}
$$
is norm dense in $ L(E,F)$. We will apply this result to prove the following. 

%If $U$ is a balanced open subset of a complex Banach space $E$, a weight $v$ on $U$ is said to be \emph{radial} if $v(\lambda z)=v(z)$ for all $\lambda\in\mathbb{T}$ and $z\in U$. %For instance, any weight defined by $v(x):=g(\|x\|)$, where $g:[0,\infty)\to[0,\infty)$ is continuous, is radial.

\begin{lemma}\label{prop-1}
Let $U$ be an open subset of a complex Banach space $E$ and let $v$ be a weight on $U$. Assume that $\T\At_{G^\infty_v(U)}$ is a norm-closed set of uniformly strongly exposed points of $B_{G^\infty_v(U)}$. Then $H^\infty_{v\NA}(U,F)$ is norm dense in $H^\infty_v(U,F)$ for every complex Banach space $F$. 
\end{lemma}

\begin{proof}
By Theorem \ref{main-theo}, we have 
$$
B_{G^\infty_v(U)}=\overline{\abco}\left(\At_{G^\infty_v(U)}\right)=\overline{\co}\left(\T\At_{G^\infty_v(U)}\right).
$$
Therefore, for every Banach space $F$, the set
$$
\left\{T\in L( G^\infty_v(U),F)\colon \exists \phi\in \T\At_{G^\infty_v(U)} \, | \, \left\|T(\phi)\right\|=\left\|T\right\|\right\}
$$
is norm dense in $ L(G^\infty_v(U),F)$ by \cite[Proposition 1]{Lin-63}. It follows that $H^\infty_{v\NA}(U,F)$ is norm dense in $H^\infty_v(U,F)$. Indeed, let $\varepsilon>0$ and $f\in H^\infty_v(U,F)$. Consider $T_f\in L( G^\infty_v(U),F)$ by Theorem \ref{main-theo} and therefore there is $T\in L( G^\infty_v(U),F)$ with $\left\|T(\lambda v(z)\delta_z)\right\|=\left\|T\right\|$ for some $\lambda\in\T$ and $z\in U$ such that $\left\|T_f-T\right\|<\varepsilon$. By Theorem \ref{main-theo} again, $T=T_{f_0}$ for some $f_0\in H^\infty_v(U,F)$. Hence $f_0\in H^\infty_{v\NA}(U,F)$ since 
$$
\left\|f_0\right\|_v=\left\|T_{f_0}\right\|=\left\|T\right\|=\left\|T(\lambda v(z)\delta_z)\right\|=\left\|T_{f_0}(\lambda v(z)\delta_z)\right\|=v(z)\left\|f_0(z)\right\|.
$$
and, finally, note that 
$$
\left\|f-f_0\right\|_v=\left\|T_{f-f_0}\right\|=\left\|T_f-T_{f_0}\right\|=\left\|T_f-T\right\|<\varepsilon .
$$
\end{proof}

We are now ready to state the main result of this paper.

\begin{theorem}\label{teo-2}
Let $U$ be an open subset of a complex Banach space $E$ and let $v$ be a weight on $U$. Assume that $\T\At_{G^\infty_v(U)}$ is a norm-closed set of uniformly strongly exposed points of $B_{G^\infty_v(U)}$. Then $H^\infty_v(U,F)$ has the $WH^\infty$-BPB property for every complex Banach space $F$.
\end{theorem}

\begin{proof}
Let $0<\varepsilon<1$. Since $\T\At_{G^\infty_v(U)}$ is a set of uniformly strongly exposed points of $B_{G^\infty_v(U)}$, there exists a set $\{f_{(\lambda,z)}\colon \lambda\in\T,\, z\in U\}\subseteq S_{H^\infty_v(U)}$ with $\lambda v(z)f_{(\lambda,z)}(z)=J_v(f_{(\lambda,z)})(\lambda v(z)\delta_z)=1$ for every $\lambda\in\T$ and $z\in U$, and a number $0<\delta<1$ such that
$$
\sup\left\{\diam(S(B_{G^\infty_v(U)},J_v(f_{(\lambda,z)}),\delta))\colon \lambda\in\T,\, z\in U\right\}<\varepsilon .
$$
Take $0<\eta<\varepsilon$ so that 
$$
\left(1+\frac{\varepsilon}{4}\right)(1-\eta)>1+\frac{\varepsilon(1-\delta)}{4}
$$
and consider $f\in S_{H^\infty_v(U,F)}$, $\lambda\in\T$ and $z\in U$ such that $v(z)\left\|f(z)\right\|>1-\eta$. Define $g_0\colon U\to F$ by 
$$
g_0(y)=f(y)+\frac{\varepsilon}{4}\lambda f_{(\lambda,z)}(y)v(y)f(z)\qquad (y\in U).
$$
%Note that $g_0\in H^\infty_v(U,F)$ since 
%$$
%v(y)\left\|g_0(y)\right\|\leq v(y)\left\|f(y)\right\|+\frac{\varepsilon}{4}v(y)\left|f_{(\lambda,z)}(y)\right|v(z)\left\|f(z)\right\|\leq 1+\frac{\varepsilon}{4}
%$$
%for all $y\in U$. 
Clearly, $g_0\in H^\infty_v(U,F)$ with $\left\|f-g_0\right\|_v\leq\varepsilon/4$ since 
$$
v(y)\left\|f(y)-g_0(y)\right\|=\frac{\varepsilon}{4}v(y)\left|f_{(\lambda,z)}(y)\right|v(z)\left\|f(z)\right\|\leq\frac{\varepsilon}{4}
$$ 
for all $y\in U$. Given $y\in U$, we claim that $v(y)\delta_y\in \T S(B_{G^\infty_v(U)},J_v(f_{(\lambda,z)}),\delta)$ whenever $v(y)\left\|g_0(y)\right\|\geq v(z)\left\|g_0(z)\right\|$. Indeed, if $v(y)\delta_y\notin \T S(B_{G^\infty_v(U)},J_v(f_{(\lambda,z)}),\delta)$, we obtain
\begin{align*}
v(y)\left\|g_0(y)\right\|&=\left\|v(y)f(y)+\frac{\varepsilon}{4}\lambda v(y)f_{(\lambda,z)}(y)v(z)f(z)\right\|\\
                         &=\left\|v(y)f(y)+\frac{\varepsilon}{4}\lambda J_v(f_{(\lambda,z)})(v(y)\delta_y)v(z)f(z)\right\|\\
                         &\leq 1+\frac{\varepsilon}{4}\left|J_v(f_{(\lambda,z)})(v(y)\delta_y)\right|\\
												 &=1+\frac{\varepsilon}{4}\Re(J_v(f_{(\lambda,z)})(\alpha v(y)\delta_y))\leq 1+\frac{\varepsilon(1-\delta)}{4},
\end{align*}
where we have used that 
$$
\left|J_v(f_{(\lambda,z)})(v(y)\delta_y)\right|=\alpha J_v(f_{(\lambda,z)})(v(y)\delta_y)=J_v(f_{(\lambda,z)})(\alpha v(y)\delta_y)=\Re(J_v(f_{(\lambda,z)})(\alpha v(y)\delta_y))
$$
for some $\alpha\in\T$, and as we also have  
$$
v(z)\left\|g_0(z)\right\|=\left(1+\frac{\varepsilon}{4}\right)v(z)\left\|f(z)\right\|>\left(1+\frac{\varepsilon}{4}\right)(1-\eta),
$$ 
our claim follows. 

Since $\left\|g_0\right\|_v\geq v(z)\left\|g_0(z)\right\|>0$, taking $g=g_0/\left\|g_0\right\|_v$, we have  
$$
\left\|f-g\right\|_v\leq\left\|f-g_0\right\|_v+\left\|g_0-g\right\|_v=\left\|f-g_0\right\|_v+\left|\left\|g_0\right\|_v-1\right|\leq\frac{\varepsilon}{4}+\frac{\varepsilon}{4}=\frac{\varepsilon}{2}.
$$
The proof will be finished if $v(z)\left\|g(z)\right\|=\left\|g\right\|_v$. Otherwise, take 
$$
0<\varepsilon'<\min\left\{\frac{\varepsilon}{2},\left\|g\right\|_v-v(z)\left\|g(z)\right\|\right\}.
$$
Since $H^\infty_{v\NA}(U,F)$ is norm dense in $H^\infty_v(U,F)$ by Lemma \ref{prop-1}, we can take $f_0\in S_{H^\infty_v(U,F)}$ and $z_0\in U$ such that $v(z_0)\left\|f_0(z_0)\right\|=1$ and $\left\|g-f_0\right\|_v<\varepsilon'$. From the inequality
\begin{align*}
v(z_0)\left\|g(z_0)\right\|&\geq v(z_0)\left\|f_0(z_0)\right\|-\left\|f_0-g\right\|_v\geq\left\|f_0\right\|_v-\varepsilon'\\
                           &\geq \left\|f_0\right\|_v-\left(\left\|g\right\|_v-v(z)\left\|g(z)\right\|\right)=v(z)\left\|g(z)\right\|,
\end{align*}
we deduce that $v(z_0)\left\|g_0(z_0)\right\|\geq v(z)\left\|g_0(z)\right\|$ and our claim yields $v(z_0)\delta_{z_0}\in \T S(B_{G^\infty_v(U)},J_v(f_{(\lambda,z)}),\delta)$. Hence $\lambda_0v(z_0)\delta_{z_0}\in S(B_{G^\infty_v(U)},J_v(f_{(\lambda,z)}),\delta)$ for some $\lambda_0\in\T$, and thus we have 
$$
\left\|\lambda v(z)\delta_z-\lambda_0 v(z_0)\delta_{z_0}\right\|\leq\diam(S(B_{G^\infty_v(U)},J_v(f_{(\lambda,z)}),\delta))<\varepsilon
$$
because also $\lambda v(z)\delta_z\in S(B_{G^\infty_v(U)},J_v(f_{(\lambda,z)}),\delta)$. Lastly, note that 
$$
\left\|f-f_0\right\|_v\leq\left\|f-g\right\|_v+\left\|g-f_0\right\|_v <\varepsilon.
$$
\end{proof}

We now present a family of spaces $H^\infty_v$ satisfying the conditions of Theorem \ref{teo-2}. For each $p>0$, let us recall that $v_p\colon\D\to\mathbb{R}^+$ is the \emph{polynomial weight} defined by $v_p(z)=(1-|z|^2)^p$ for all $z\in\D$. 

For each $z\in\D$, let $\phi_z\colon\D\to\D$ be the \textit{M\"obius transformation} given by 
$$
\phi_z(w)=\frac{z-w}{1-\overline{z}w}\qquad (w\in\D).
$$
Consider the \emph{pseudohyperbolic metric} $\rho\colon\D\times\D\to\mathbb{R}$ defined by 
$$
\rho(z,w)=\left|\phi_z(w)\right|\qquad (z,w\in\D).
$$
It is easy to check that  
$$
\phi_z'(w)=\frac{1-|z|^2}{(1-\overline{z}w)^2}
$$
and
$$
(1-|w|^2)\left|\phi_z'(w)\right|=\frac{(1-|w|^2)(1-|z|^2)}{|1-\overline{z}w|^2}=1-\left|\frac{z-w}{1-\overline{z}w}\right|^2=1-\rho(z,w)^2
$$
for all $w\in\D$. 

We will use a reformulation of an inequality stated in \cite{HedKorZhu-00}.

%\begin{lemma}\label{lema0}\cite[Lemma 5.1]{HedKorZhu-00}
%Let $p>0$. Then there exists a constant $M_p>1$ (depending only on $p$) such that $\left\|v_p(z)\delta_z-v_p(w)\delta_w\right\|\leq M_p\,\rho(z,w)$ for all $z,w\in\D$. $\hfill\qed$
%\end{lemma}

\begin{lemma}\label{lema0}\cite[Lemma 5.1]{HedKorZhu-00}
Let $p>0$ and $f\in H^\infty_{v_p}(\D)$. Then there exists a constant $N_p>0$ (depending only on $p$) such that
$$
\left|v_p(z)f(z)-v_p(w)f(w)\right|\leq N_p\left\|f\right\|_{v_p}\rho(z,w)
$$
for all $z,w\in\D$ with $\rho(z,w)\leq 1/2$. $\hfill\qed$
\end{lemma}

In view of Lemma \ref{lema0}, we can make the following.

\begin{remark}\label{rem-1}
For any $f\in H^\infty_{v_p}(\D)$, on the one hand, we have 
$$
\left|v_p(z)f(z)-v_p(w)f(w)\right|\leq N_p\left\|f\right\|_{v_p}\rho(z,w)
$$
whenever $z,w\in\D$ with $\rho(z,w)\leq 1/2$ by Lemma \ref{lema0}, and on the other hand, we get
$$
\left|v_p(z)f(z)-v_p(w)f(w)\right|\leq 2\left\|f\right\|_{v_p}\leq 4\left\|f\right\|_{v_p}\rho(z,w)
$$
whenever $z,w\in\D$ with $\rho(z,w)>1/2$. Therefore, taking $M_p=\max\{N_p,4\}$, we infer that   
$$
\left\|v_p(z)\delta_z-v_p(w)\delta_w\right\|\leq M_p\,\rho(z,w)\qquad (z,w\in\D).
$$
\end{remark}

We will also apply the following easy fact.

\begin{lemma}\label{auxremark}
Let $\varepsilon>0$. If $\lambda\in\C$ with $|\lambda|\leq 1$ and $1-\mathrm{Re}(\lambda)<\varepsilon^2/2$, then $|1-\lambda|<\varepsilon$.$\hfill\qed$
\end{lemma}

We now have all the necessary tools to prove the following result.

\begin{theorem}\label{prop-vp}
$H^\infty_{v_p}(\D,F)$ with $p\geq 1$ has the $WH^\infty$-BPB property for every complex Banach space $F$.
\end{theorem}

\begin{proof}
Fix $p\geq 1$. Let $\lambda\in\T$ and $z\in\D$. Define the function $f_z\colon\D\to\C$ by 
$$
f_z(w)=(\phi_z'(w))^p\qquad (w\in\D).
$$ 
Clearly, $\overline{\lambda}f_z\in H(\D)$ and since
$$
(1-|w|^2)^p|(\overline{\lambda}f_z)(w)|=(1-|w|^2)^p|\phi_z'(w)|^p=\left(1-\rho(z,w)^2\right)^p\leq 1
$$
for all $w\in\D$, it follows that $\overline{\lambda}f_z\in H^\infty_{v_p}(\D)$ with $\left\|\overline{\lambda}f_z\right\|_{v_p}\leq 1$. Hence $J_{v_p}(\overline{\lambda}f_z)\in G^\infty_{v_p}(\D)^*$ with $\left\|J_{v_p}(\overline{\lambda}f_z)\right\|\leq 1$ by Theorem \ref{main-theo}. In fact,  
$$
J_{v_p}(\overline{\lambda}f_z)(\lambda v_p(z)\delta_z)=\lambda v_p(z)\delta_z(\overline{\lambda}f_z)=\lambda v_p(z)\overline{\lambda}f_z(z)=v_p(z)f_z(z)=1
$$
where $\lambda v_p(z)\delta_z\in G^\infty_{v_p}(\D)$ with $\left\|\lambda v_p(z)\delta_z\right\|=1$, and thus $\left\|J_{v_p}(\overline{\lambda}f_z)\right\|=1$.

%We will prove that the set of functionals 
%$$
%\left\{J_{v_p}(\overline{\lambda}f_z)\colon \lambda\in\T,\, z\in\D\right\}\subseteq S_{G^\infty_{v_p}(\D)^*}
%$$ 
%uniformly strongly exposes the points of $B_{G^\infty_{v_p}(\D)}$, that is, 
Now, we will prove that for every $0<\varepsilon<1$ there exists $0<\delta<1$ such that 
$$
\mathrm{Re}(J_{v_p}(\overline{\lambda}f_z)(\phi))>1-\delta,\; \phi\in B_{G^\infty_{v_p}(\D)}\quad\Rightarrow\quad \left\|\phi-\lambda v_p(z)\delta_z\right\|<\varepsilon.
$$
Let $0<\varepsilon<1$. Remark \ref{rem-1} provides a constant $M_p>1$ so that    
$$
\left\|v_p(z)\delta_z-v_p(w)\delta_w\right\|\leq M_p\,\rho(z,w)\qquad (z,w\in\D).
$$
Take $\delta_1=\varepsilon/6M_p$ and we have 
$$
\left\|v_p(z)\delta_z-v_p(w)\delta_w\right\|<\frac{\varepsilon}{6}\qquad (z,w\in\D,\; \rho(z,w)<\delta_1).
$$
Let $\delta=(\varepsilon/36)\left(\min\left\{\delta_1^2,\varepsilon/6\right\}\right)^2$. We claim that 
$$
\mathrm{Re}(J_{v_p}(\overline{\lambda}f_z)(\alpha v_p(w)\delta_w))>1-\frac{18\delta}{\varepsilon},\; \alpha\in\T ,\; w\in\D\quad\Rightarrow\quad \left\|\alpha v_p(w)\delta_w-\lambda v_p(z)\delta_z\right\|<\frac{\varepsilon}{2}.
$$
Indeed, let $w\in\D$ and $\alpha\in\T$ with $\mathrm{Re}(J_{v_p}(\overline{\lambda}f_z)(\alpha v_p(w)\delta_w))>1-18\delta/\varepsilon$. It is clear that 
$$
1-\mathrm{Re}\left(\alpha v_p(w)\delta_w\left(\overline{\lambda}f_z\right)\right)<\frac{18\delta}{\varepsilon}=\frac{1}{2}\left(\min\left\{\delta_1^2,\frac{\varepsilon}{6} \right\}\right)^2
$$
and 
$$
\left|\alpha v_p(w)\delta_w\left(\overline{\lambda}f_z\right)\right|\leq \left\|v_p(w)\delta_w\right\|\left\|f_z\right\|_{v_p}=1,
$$
and therefore Lemma \ref{auxremark} yields 
$$
\left|1-\alpha v_p(w)\delta_w\left(\overline{\lambda}f_z\right)\right|<\min\left\{\delta_1,\frac{\varepsilon}{6}\right\}.
$$
Thus, by the properties of $f_z$ and the fact that $p\geq 1$, it follows that 
\begin{align*}
\rho(z,w)^2&\leq 1-\left(1-|w|^2\right)^p|f_z(w)|\leq \left|1-\alpha\overline{\lambda}\left(1-|w|^2\right)^pf_z(w) \right|\\
&=\left|1-\alpha v_p(w)\delta_w\left(\overline{\lambda}f_z\right)\right|<\min\left\{\delta_1^2,\frac{\varepsilon}{6}\right\},
\end{align*}
so $\rho(z,w)<\delta_1$. Therefore, $\left\|v_p(w)\delta_w-v_p(z)\delta_z\right\|<\varepsilon/6$. Furthermore, we have 
\begin{align*}
|\lambda-\alpha|&\leq\left|\lambda-\alpha v_p(w)\delta_w(f_z)\right|+\left|\alpha v_p(w)\delta_w(f_z)-\alpha\right|\\
&=\left|1-\alpha v_p(w)\delta_w\left(\overline{\lambda}f_z\right)\right|+\left|v_p(w)\delta_w(f_z)-v_p(z)\delta_z(f_z)\right|\\
&<\frac{\varepsilon}{6}+\left\|v_p(w)\delta_w-v_p(z)\delta_z\right\|<\frac{\varepsilon}{3}.
\end{align*}
Hence
\begin{align*}
\left\|\alpha v_p(w)\delta_w-\lambda v_p(z)\delta_z\right\|&\leq\left\|\alpha v_p(w)\delta_w-\alpha v_p(z)\delta_z\right\|+\left\|\alpha v_p(z)\delta_z-\lambda v_p(z)\delta_z\right\|\\
&<\frac{\varepsilon}{6}+|\alpha-\lambda|\left\|v_p(z)\delta_z\right\|<\frac{\varepsilon}{2},
\end{align*}
and this proves our claim.

Now, let $\phi\in B_{G^\infty_{v_p}(\D)}$ such that $\mathrm{Re}(J_{v_p}(\overline{\lambda}f_z)(\phi))>1-\delta$. We will show that $\left\|\phi-\lambda v_p(z)\delta_z\right\|<\varepsilon$. Since $B_{G^\infty_{v_p}(\D)}=\overline{\co}\left(\T\At_{G^\infty_{v_p}(\D)}\right)$ by Theorem \ref{main-theo}, then there exists $\gamma\in\co\left(\T\At_{G^\infty_{v_p}(\D)}\right)$ such that $\|\phi-\gamma\|< \min\left\{\varepsilon/6,\delta\right\}$. Thus
$$
1-\mathrm{Re}\left(\gamma\left(\overline{\lambda}f_z\right)\right)=1-\mathrm{Re}\left(\phi\left(\overline{\lambda}f_z\right)\right)+\mathrm{Re}\left(\phi\left(\overline{\lambda}f_z\right)-\gamma\left(\overline{\lambda}f_z\right)\right)<2\delta.
$$
Let $z_1,\ldots,z_m\in\D$, $\lambda_1,\ldots,\lambda_m\in\T$ and $t_1,\ldots,t_m\in [0,1]$ be with $\sum_{j=1}^mt_j=1$ such that $\gamma=\sum_{j=1}^m t_j\lambda_jv_p(z_j)\delta_{z_j}$. Fix
$$
I=\left\{j\in\{1,\ldots, m\}\colon v_p(z_j)\mathrm{Re}\left(\lambda_j\delta_{z_j}\left(\overline{\lambda}f_z\right)\right)< 1-\frac{15\delta}{\varepsilon} \right\}.
$$
On the one hand, we have
\begin{align*}
\mathrm{Re}\left(\gamma\left(\overline{\lambda}f_z\right)\right)&=\sum_{j=1}^m t_jv_p(z_j)\mathrm{Re}\left(\lambda_j\delta_{z_j}\left(\overline{\lambda}f_z\right)\right)\\
&\leq\sum_{j\in\{1,\ldots,m\}\backslash I}t_j+\left(\sum_{j\in I}t_j\right)\left(1-\frac{15\delta}{\varepsilon}\right)
=1-\frac{15\delta}{\varepsilon}\sum_{j\in I}t_j,
\end{align*}
and it follows that 
$$
\sum_{j\in I}t_j\leq \frac{\varepsilon}{15\delta} 
\left(1-\mathrm{Re}\left(\gamma\left(\overline{\lambda}f_z\right)\right)\right)<
\frac{\varepsilon}{15\delta} 2\delta<\frac{\varepsilon}{6}.
$$
On the other hand, given $j\in\{1,\ldots,m\}\backslash I$, we have 
$$
1-\frac{18\delta}{\varepsilon}<1-\frac{15\delta}{\varepsilon}
\leq v_p(z_j)\mathrm{Re}\left(\lambda_j\delta_{z_j}\left(\overline{\lambda}f_z\right)\right)
=\mathrm{Re}(J_{v_p}(\overline{\lambda}f_z)(\lambda_j v_p(z_j)\delta_{z_j})),
$$
and our previous claim yields $\left\|\lambda_jv_p(z_j)\delta_{z_j}-\lambda v_p(z)\delta_z\right\|<\varepsilon/2$.
Then
\begin{align*}
\left\|\phi-\lambda v_p(z)\delta_z\right\| &\leq \|\phi-\gamma\|+\left\|\gamma-\lambda v_p(z)\delta_z\right\|\\
&<\frac{\varepsilon}{6}+\left\|\sum_{j=1}^mt_j\lambda_j v_p(z_j)\delta_{z_j}-\lambda v_p(z)\delta_z\right\|\\
&\leq \frac{\varepsilon}{6}+\sum_{j\in I}t_j\left\|\lambda_j v_p(z_j)\delta_{z_j}-\lambda v_p(z)\delta_z\right\|+\sum_{j\in \{1,\ldots,m\}\backslash I}t_j\left\|\lambda_jv_p(z_j)\delta_{z_j}
-\lambda v_p(z)\delta_z\right\|\\
&\leq \frac{\varepsilon}{6}+2\sum_{j\in I}t_j+\frac{\varepsilon}{2}\sum_{j\in \{1,\ldots,m\}\backslash I}t_j\\
&<\frac{\varepsilon}{6}+2\frac{\varepsilon}{6}+\frac{\varepsilon}{2}=\varepsilon,
\end{align*}
as required.

Finally, we will prove that $\T\At_{G^\infty_{v_p}(\D)}$ is norm-closed in $ G^\infty_{v_p}(\D)$. For it, let $(\lambda_n)$ and $(z_n)$ be sequences in $\T$ and $\D$, respectively, such that $(\lambda_n v_p(z_n)\delta_{z_n})$ converges in norm to some $\phi\in G^\infty_{v_p}(\D)$. This implies that $\left\|\phi\right\|=1$. Since $\T$ and $\overline{\D}$ are compact, we can take subsequences $(\lambda_{n_k})_k$ and $(z_{n_k})_k$ which converge to some $\lambda_0\in\T$ and $z_0\in\overline{\D}$, respectively. If $z_0\in\T$, we have    
$$
\phi=\lim_{k\to\infty}\lambda_{n_k}v_p(z_{n_k})\delta_{z_{n_k}}=0,
$$
which is impossible. Hence $z_0\in\D$, and we conclude that     
$$
\phi=\lim_{k\to\infty}\lambda_{n_k}v_p(z_{n_k})\delta_{z_{n_k}}=\lambda_0 v_p(z_0)\delta_{z_0}.
$$
\end{proof}

%%%%%%%%%%%%%%%%%%%%%%%%%%%%%%%%%%%%%%%%%%%%%%%%%%%%%%%%%%%%%%%%%%%%%%%%%%%%%%%%%%%%%%%%%%%%%%%%%%%%%%%%%%%%%%%%%%%%%%%%%%%%%%%%%%%%%%%%%%%%%%%%%%%%%%%%%%%%%%%%%%%%%%%%%%%%%%%%%%%%%%%%%%%%%%%%%%%%%%%%%%%%

\section{Relationship between the complex and vector-valued cases of the $WH^\infty$-BPB property}\label{s2}

Our goal in this section is to study when the $WH^\infty$-BPB property for vector-valued weighted holomorphic mappings is inherited from the $WH^\infty$-BPB property for complex-valued weighted holomorphic functions, and vice versa.

\begin{proposition}\label{prop-3}
Let $U$ be an open subset of a complex Banach space $E$ and let $v$ be a weight on $U$. Suppose that there exists a non-zero complex Banach space $F$ such that $H^\infty_v(U,F)$ has the $WH^\infty$-BPB the property. Then  $H^\infty_v(U)$ has the $WH^\infty$-BPB property.  
\end{proposition}

\begin{proof}
Let $0<\varepsilon<1$ and assume that $H^\infty_v(U,F)$ has the $WH^\infty$-BPB property with function $\varepsilon\mapsto\eta(\varepsilon)$. Take $h\in S_{H^\infty_v(U)}$, $\lambda\in\T$ and $z\in U$ such that $v(z)\left|h(z)\right|>1-\eta(\varepsilon/2)$. Pick $y_0\in S_F$ and define $f\in H^\infty_v(U,F)$ by $f(x)=h(x)y_0$ for all $x\in U$. Clearly, $f\in S_{H^\infty_v(U,F)}$ with $v(z)\left\|f(z)\right\|>1-\eta(\varepsilon/2)$. By hypothesis, we can find $f_0\in S_{H^\infty_v(U,F)}$, $\lambda_0\in\T$ and $z_0\in U$ such that $v(z_0)\left\|f_0(z_0)\right\|=1$, $\left\|f-f_0\right\|_v<\varepsilon/2$ and $\left\|\lambda v(z)\delta_z-\lambda_0v(z_0)\delta_{z_0}\right\|<\varepsilon/2$. Now take $y^*\in S_{F^*}$ such that $y^*(v(z_0)f_0(z_0))=1$. We have 
$$
\left\|y^*(y_0)h-y^*\circ f_0\right\|_v=\left\|y^*\circ f-y^*\circ f_0\right\|_v\leq \left\|y^*\right\|\left\|f-f_0\right\|_v<\frac{\varepsilon}{2},
$$
and this implies 
$$
1-\left|y^*(y_0)\right|=\left|\left|y^*(y_0)\right|-1\right|=\left|\left|y^*(y_0)\right|\left\|h\right\|_v-\left\|y^*\circ f_0\right\|_v\right|\leq\left\|y^*(y_0)h-y^*\circ f_0\right\|_v<\frac{\varepsilon}{2}.
$$
We can write $\left|y^*(y_0)\right|=\alpha_0 y^*(y_0)$ for some $\alpha_0\in\T$. Take $z^*=\alpha_0 y^*$. Clearly, $z^*\in S_{F^*}$ with $0\geq 1-z^*(y_0)<\varepsilon/2$. Furthermore,  
$$
\left\|z^*(y_0)h-z^*\circ f_0\right\|_v=\left\|y^*(y_0)h-y^*\circ f_0\right\|_v<\frac{\varepsilon}{2},
$$
and  
$$
\left\|h-z^*(y_0)h\right\|_v=\left|1-z^*(y_0)\right|\left\|h\right\|_v=1-z^*(y_0)<\frac{\varepsilon}{2}.
$$
Therefore, writing $h_0=z^*\circ f_0$, we conclude that $h_0\in S_{H^\infty_v(U)}$ with $v(z_0)\left|h_0(z_0)\right|=1$ and 
$$
\left\|h-h_0\right\|_v\leq\left\|h-z^*(y_0)h\right\|_v+\left\|z^*(y_0)h-h_0\right\|_v<\frac{\varepsilon}{2}+\frac{\varepsilon}{2}=\varepsilon .
$$
\end{proof}

We can establish a similar result for the density of weighted holomorphic mappings attaining their norms.

\begin{proposition}\label{prop-33}
Let $U$ be an open subset of a complex Banach space $E$ and let $v$ be a weight on $U$. Suppose that there exists a non-zero complex Banach space $F$ such that $H^\infty_{v\NA}(U,F)$ is norm dense in $H^\infty_v(U,F)$. Then $H^\infty_{v\NA}(U)$ is norm dense in $H^\infty_v(U)$.
\end{proposition}

\begin{proof}
Let $\varepsilon>0$ and $h\in S_{H^\infty_v(U)}$. Pick $y_0\in S_F$ and define $f\in H^\infty_v(U,F)$ by $f(x)=h(x)y_0$ for all $x\in U$. By hypothesis, we can find $f_0\in S_{H^\infty_v(U,F)}$ and $z_0\in U$ such that $v(z_0)\left\|f_0(z_0)\right\|=1$ and $\left\|f-f_0\right\|_v<\varepsilon/2$. Now take $y^*\in S_{F^*}$ such that $y^*(v(z_0)f_0(z_0))=1$. Let $\alpha_0\in\T$ be so that $\left|y^*(y_0)\right|=\alpha_0 y^*(y_0)$ and take $z^*=\alpha_0 y^*\in S_{F^*}$. As in the proof of Proposition \ref{prop-33}, we have $\left\|z^*(y_0)h-z^*\circ f_0\right\|_v<\varepsilon/2$ and $\left\|z^*(y_0)h-h\right\|_v<\varepsilon/2$. Therefore, $h_0=z^*\circ f_0\in S_{H^\infty_v(U)}$ with $v(z_0)\left|h_0(z_0)\right|=1$, hence $h_0\in H^\infty_{v\NA}(U)$) and satisfies that
$$
\left\|h-h_0\right\|_v\leq\left\|h-z^*(y_0)h\right\|_v+\left\|z^*(y_0)h-h_0\right\|_v<\frac{\varepsilon}{2}+\frac{\varepsilon}{2}=\varepsilon .
$$
\end{proof}

Our next aim is to study the converse problem of passing from the $WH^\infty$-BPB property for complex-valued weighted holomorphic functions to vector-valued mappings. We will need the following concept 
introduced by Lindenstrauss \cite{Lin-63} and renamed as the property $\beta$ by Schachermayer \cite{Sch-83}.

\begin{definition}\cite{Lin-63}\label{def-property-beta}
A Banach space $F$ has \emph{the property $\beta$} if there is a set $\left\{(y_i,y^*_i)\colon i\in I\right\}\subseteq F\times F^*$, and a constant $0\leq\rho<1$ satisfying the following properties:
\begin{enumerate}
	\item $\left\|y^*_i\right\|=\left\|y_i\right\|=y^*_i(y_i)=1$ for every $i\in I$.
	\item $\left|y^*_i(y_j)\right|\leq\rho$ for every $i,j\in I$ with $i\neq j$.
  \item $\left\|y\right\|=\sup\left\{\left|y^*_i(y)\right|\colon i\in I\right\}$ for every $y\in F$.
\end{enumerate}
\end{definition}

Examples of Banach spaces with the property $\beta$ are the finite-dimensional spaces whose unit ball is a polyhedron, the sequence spaces $c_0$ and $\ell_1$ endowed with their usual norms, and those spaces of continuous functions $C(K)$ where $K$ is a compact Hausdorff topological space having a dense set of isolated points. Besides, Partington \cite{Par-82} proved that every Banach space admits an equivalent norm with this property.

Our arguments will require a version for weighted holomorphic mappings of the concept of the adjoint operator $T^*\in L(F^*,E^*)$ of an operator $T\in L(E,F)$. Let $f\in H^\infty_v(U,F)$. Given $y^*\in F^*$, it is clear that $y^*\circ f\in H(U)$ and   
$$
v(z)\left\|(y^*\circ f)(z)\right\|=v(z)\left\|y^*(f(z))\right\|\leq v(z)\left\|y^*\right\|\left\|f(z)\right\|\leq \left\|f\right\|_v\left\|y^*\right\|
$$
for all $z\in U$. Hence $y^*\circ f\in H^\infty_v(U)$ with $\left\|y^*\circ f\right\|_v\leq \left\|f\right\|_v\left\|y^*\right\|$. This justifies the following.

\begin{definition}
Let $U$ be an open subset of a complex Banach space $E$, $v$ be a weight on $U$ and $F$ be a complex Banach space. Given $f\in H^\infty_v(U,F)$, the \textit{weighted holomorphic transpose} of $f$ is the mapping $f^t\colon F^*\to H^\infty_v(U)$ given by $f^t(y^*)=y^*\circ f$ for all $y^*\in F^*$.
\end{definition}

Clearly, $f^t$ is linear and continuous with $||f^t||\leq \left\|f\right\|_v$. Furthermore, $||f^t||=\left\|f\right\|_v$. Indeed, for $0<\varepsilon<\left\|f\right\|_v$, take $z\in U$ such that $v(z)\left\|f(z)\right\|>\left\|f\right\|_v-\varepsilon$. By Hahn--Banach Theorem, there exists $y^*\in S_{F^*}$ such that $\left|y^*(f(z))\right|=\left\|f(z)\right\|$. We have 
$$
\left\|f^t\right\|\geq\sup_{x^*\neq 0}\frac{\left\|f^t(x^*)\right\|_v}{\left\|x^*\right\|}
\geq\frac{\left\|y^*\circ f\right\|_v}{\left\|y^*\right\|}
\geq v(z)\left|y^*(f(z))\right|=v(z)\left\|f(z)\right\|>\left\|f\right\|_v-\varepsilon .
$$
Letting $\varepsilon\to 0$, one obtains $||f^t||\geq \left\|f\right\|_v$, as desired. Finally, note that  
\begin{align*}
(J_v\circ f^t)(y^*)(v(z)\delta_z)&=J_v(f^t(y^*))(v(z)\delta_z)=J_v(y^*\circ f)(v(z)\delta_z)\\
                                 &=v(z)(y^*\circ f)(z)=v(z)y^*(f(z))\\
                                 &=y^*(T_f(v(z)\delta_z))=(T_f)^*(y^*)(v(z)\delta_z)
\end{align*}
for all $y^*\in F^*$ and $z\in U$, where $(T_f)^*\colon F^*\to G^\infty_v(U)^*$ is the adjoint operator of $T_f$. Since $ G^\infty_v(U)=\overline{\lin}(\At_{G^\infty_v(U)})$, we deduce that $J_v\circ f^t=(T_f)^*$. So we have proved the following.

\begin{proposition}\label{prop-A}
Let $U$ be an open subset of a complex Banach space $E$, $v$ be a weight on $U$ and $F$ be a complex Banach space. If $f\in H^\infty_v(U,F)$, then $f^t\in L(F^*, H^\infty_v(U))$ with $||f^t||=\left\|f\right\|_v$ and $f^t=J_v^{-1}\circ(T_f)^*$. $\hfill\Box$
\end{proposition}

We make a brief pause in our study on the $WH^\infty$-BPB property to show that this transposition permits us to identify the spaces $H^\infty_v(U,F)$ and $H^\infty_{vK}(U,F)$ with certain distinguished subspaces of operators. %We first see that the mapping $f\mapsto f^t$ identifies $H^\infty_v(U,F)$ with the subspace of $L(F^*,H^\infty_v(U))$ formed by all weak*-to-weak* continuous linear operators from $F^*$ into $H^\infty_v(U)$.

\begin{proposition}\label{teo-4-1}
Let $U$ be an open subset of a complex Banach space $E$, $v$ be a weight on $U$ and $F$ be a complex Banach space. Then $f\mapsto f^t$ is an isometric isomorphism from $H^\infty_v(U,F)$ onto $L((F^*,w^*);(H^\infty_v(U),w^*))$.
\end{proposition}

\begin{proof}
Let $f\in H^\infty_v(U,F)$. Hence $f^t=J_v^{-1}\circ (T_f)^*\in L((F^*,w^*);(H^\infty_v(U),w^*))$. 
%by Theorem \ref{main-theo} and \cite[Theorem 3.1.11]{Meg-98}. 
We have $||f^t||=\left\|f\right\|_v$ by Proposition \ref{prop-A}. To show the surjectivity of the mapping in the statement, let $T\in L((F^*,w^*);(H^\infty_v(U),w^*))$. Then the mapping $J_v\circ T$ is in $L((F^*,w^*);(G^\infty_v(U)^*,w^*))$ and therefore there is a $S\in L(G^\infty_v(U),F)$ such that $S^*=J_v\circ T$. By Theorem \ref{main-theo}, there exists $f\in H^\infty_v(U,F)$ such that $T_f=S$, and thus $T=J_v^{-1}\circ(T_f)^*=f^t$, as desired. 
\end{proof}

The next result contains a version of Schauder Theorem for mappings in $H^\infty_{vK}(U,F)$.

\begin{theorem}\label{teo-4-2}
Let $U$ be an open subset of a complex Banach space $E$, $v$ be a weight on $U$ and $F$ be a complex Banach space. For any $f\in H^\infty_v(U,F)$, the following assertions are equivalent:
\begin{enumerate}
  \item $f\in H^\infty_{vK}(U,F)$.
	\item $f^t\colon F^*\to H^\infty_v(U)$ is compact.
	\item $f^t\colon F^*\to H^\infty_v(U)$ is bounded-weak*-to-norm continuous. 
	\item $f^t\colon F^*\to H^\infty_v(U)$ is compact and bounded-weak*-to-weak continuous. 
	\item $f^t\colon F^*\to H^\infty_v(U)$ is compact and weak*-to-weak continuous.
\end{enumerate}
\end{theorem}

\begin{proof}
$(i)\Leftrightarrow (ii)$: applying Theorem \ref{main-theo}, the Schauder Theorem, and the ideal property of compact operators between Banach spaces, %\cite[Proposition 3.4.10]{Meg-98}, 
we have 
\begin{align*}
f\in H^\infty_{vK}(U,F)
&\Leftrightarrow T_f\in K(G^\infty_v(U),F)\\
&\Leftrightarrow (T_f)^*\in K(F^*,G^\infty_v(U)^*)\\
&\Leftrightarrow f^t=J_v^{-1}\circ(T_f)^*\in K(F^*, H^\infty_v(U)).
\end{align*}
$(i)\Leftrightarrow (iii)$: similarly,  
\begin{align*}
f\in H^\infty_{vK}(U,F)
&\Leftrightarrow T_f\in K(G^\infty_v(U),F)\\
&\Leftrightarrow (T_f)^*\in L((F^*,bw^*); G^\infty_v(U)^*)\\
&\Leftrightarrow f^t=J_v^{-1}\circ(T_f)^*\in L((F^*,bw^*); H^\infty_v(U)),
\end{align*}
by Theorem \ref{main-theo} and \cite[Theorem 3.4.16]{Meg-98}. 

$(iii)\Leftrightarrow (iv)\Leftrightarrow (v)$: it follows from \cite[Proposition 3.1]{Kim-13}.
\end{proof}

%We also can identify $H^\infty_{vK}(U,F)$ with the subspace of $L((F^*,w^*);(G^\infty_v(U),w^*))$ consisting of all bounded-weak*-to-norm continuous linear operators from $F^*$ into $ H^\infty_v(U)$.

\begin{proposition}\label{cor-4-1}
Let $U$ be an open subset of a complex Banach space $E$, $v$ be a weight on $U$ and $F$ be a complex Banach space. Then $f\mapsto f^t$ is an isometric isomorphism from $H^\infty_{vK}(U,F)$ onto $ L((F^*,bw^*); H^\infty_v(U))$.
\end{proposition}

\begin{proof}
Let $f\in H^\infty_{vK}(U,F)$. Then $f^t\in L((F^*,bw^*); H^\infty_v(U))$ by Theorem \ref{teo-4-2} and $||f^t||=\left\|f\right\|_v$ by Proposition \ref{prop-A}. To prove the surjectivity, take $T\in L((F^*,bw^*); H^\infty_v(U))$. Then $J_v\circ T\in L((F^*,bw^*); G^\infty_v(U)^*)$. If $Q_{G^\infty_v(U)}$ denotes the natural injection from $G^\infty_v(U)$ into $G^\infty_v(U)^{**}$, then $Q_{G^\infty_v(U)}(\phi)\circ J_v\circ T\in L((F^*,bw^*);\C)$ for all $\phi\in G^\infty_v(U)$ and, by \cite[Theorem 2.7.8]{Meg-98}, $Q_{G^\infty_v(U)}(\phi)\circ J_v\circ T\in L((F^*,w^*);\C)$ for all $\phi\in G^\infty_v(U)$, that is, $J_v\circ T\in L((F^*,w^*);(G^\infty_v(U)^*,w^*))$ by \cite[Corollary 2.4.5]{Meg-98}. Hence $ J_v\circ T=S^*$ for some $S\in L(G^\infty_v(U),F)$ by \cite[Theorem 3.1.11]{Meg-98}. Note that $S^*\in L((F^*,bw^*);G^\infty_v(U)^*)$ and this means that $S\in K(G^\infty_v(U),F)$ by \cite[Theorem 3.4.16]{Meg-98}. Now, $S=T_f$ for some $f\in H^\infty_{vK}(U,F)$ by Theorem \ref{main-theo}. Finally, we have $T=J_v^{-1}\circ S^*= J_v^{-1}\circ (T_f)^*=f^t$.
\end{proof}

Returning to the $WH^\infty$-BPB property, we can adapt the proof of \cite[Theorem 2.2]{AcoAroGarMae-08} to yield the next result in the weighted holomorphic setting.

\begin{theorem}\label{prop-beta}
Let $U$ be an open subset of a complex Banach space $E$ and let $v$ be a weight on $U$. Suppose that $H^\infty_v(U)$ has the $WH^\infty$-BPB property and let $F$ be a complex Banach space satisfying the property $\beta$. Then $H^\infty_v(U,F)$ has the $WH^\infty$-BPB property. 
\end{theorem}

\begin{proof}
By hypothesis, $H^\infty_v(U)$ has the $WH^\infty$-BPB property by a function $\varepsilon\mapsto\eta(\varepsilon)$. Take a set $\left\{(y_i,y^*_i)\colon i\in I\right\}\subseteq F\times F^*$ and a number $0\leq\rho<1$ satisfying Definition \ref{def-property-beta}. Let $0<\varepsilon<1$ and choose $0<\gamma<\varepsilon/4$ so that 
$$
1+\rho\left(\frac{\varepsilon}{2}+2\gamma\right)<\left(1+\frac{\varepsilon}{2}\right)(1-3\gamma).
$$
Consider $f\in S_{H^\infty_v(U,F)}$, $\lambda\in\T$ and $z\in U$ such that $v(z)\left\|f(z)\right\|>1-\eta(\gamma)$. Take $i\in I$ so that $v(z)\left|y^*_i(f(z))\right|>1-\eta(\gamma)$. Note that $f^t(y^*_i)\in H^\infty_v(U)$ with 
$$
\left\|f^t(y^*_i)\right\|_v\geq v(z)\left|y^*_i(f(z))\right|>1-\eta(\gamma)>1-\gamma .
$$
By hypothesis there exist $g_0\in S_{H^\infty_v(U)}$, $\lambda_0\in\T$ and $z_0\in U$ such that $\left\|g_0-f^t(y^*_i)/\left\|f^t(y^*_i)\right\|_v\right\|_v<\gamma$, $v(z_0)\left|g_0(z_0)\right|=1$ and $\left\|\lambda v(z)\delta_z-\lambda_0 v(z_0)\delta_{z_0}\right\|<\gamma$. Hence we have 
$$
\left\|g_0-f^t(y^*_i)\right\|_v\leq \left\|g_0-\frac{f^t(y^*_i)}{\left\|f^t(y^*_i)\right\|_v}\right\|_v+\left\|\frac{f^t(y^*_i)}{\left\|f^t(y^*_i)\right\|_v}-f^t(y^*_i)\right\|_v
<\gamma+\left(1-\left\|f^t(y^*_i)\right\|_v\right)<2\gamma .
$$
Define $g\colon U\to F$ by 
$$
g(y)=f(y)+\left[\left(1+\frac{\varepsilon}{2}\right)g_0(y)-f^t(y^*_i)(y)\right]y_i\qquad (y\in U).
$$
Clearly, $g\in H^\infty_v(U,F)$ with $\left\|f-g\right\|_v<\varepsilon$ since 
$$
v(y)\left\|f(y)-g(y)\right\|\leq v(y)\frac{\varepsilon}{2}\left|g_0(y)\right|+v(y)\left|g_0(y)-f^t(y^*_i)(y)\right|\leq\frac{\varepsilon}{2}+2\gamma<\varepsilon
$$ 
for all $y\in U$. We will prove that $g$ attains its weighted supremum uniform norm at $z_0$. Note that 
$$
g^t(y^*)=f^t(y^*)+y^*(y_i)\left(\frac{\varepsilon}{2}g_0+g_0-f^t(y^*_i)\right)\qquad (y^*\in F^*),
$$
and using the properties of $\left\{(y_i,y^*_i)\colon i\in I\right\}$ in Definition \ref{def-property-beta}, it holds that 
\begin{align*}
\left\|g^t\right\|&=\sup_{y^*\in S_{F^*}}\left\|g^t(y^*)\right\|_v=\sup_{y^*\in S_{F^*}}\sup_{y\in U}v(y)\left|g^t(y^*)(y)\right|=\sup_{y\in U}\sup_{y^*\in S_{F^*}}\left|y^*(v(y)g(y))\right|\\
&=\sup_{y\in U}\left\|v(y)g(y)\right\|=\sup_{y\in U}\sup_{i\in I}\left|y^*_i(v(y)g(y))\right|=\sup_{i\in I}\sup_{y\in U}v(y)\left|g^t(y^*_i)(y)\right|=\sup_{i\in I}\left\|g^t(y^*_i)\right\|_v.
\end{align*}
Given $j\in I$ with $j\neq i$, we have 
$$
\left\|g^t(y^*_j)\right\|_v\leq 1+\rho\left(\frac{\varepsilon}{2}+2\gamma\right).
$$
Moreover, $g^t(y^*_i)=(1+\varepsilon/2)g_0$ and thus
$$
\left\|g^t(y^*_i)\right\|_v=\left(1+\frac{\varepsilon}{2}\right)\left\|g_0\right\|_v>\left(1+\frac{\varepsilon}{2}\right)(1-3\gamma)>1+\rho\left(\frac{\varepsilon}{2}+2\gamma\right).
$$
Therefore, $\left\|g^t\right\|=\left\|g^t(y^*_i)\right\|_v$. By Proposition \ref{prop-A}, we deduce 
\begin{align*}
\left\|g\right\|_v&=\left\|g^t\right\|=\left\|g^t(y^*_i)\right\|_v=\left(1+\frac{\varepsilon}{2}\right)\left\|g_0\right\|_v=\left(1+\frac{\varepsilon}{2}\right)v(z_0)\left|g_0(z_0)\right|\\
&=v(z_0)\left|y^*_i(g(z_0))\right|=\left|y^*_i(v(z_0)g(z_0))\right|\leq v(z_0)\left\|g(z_0)\right\|\leq\left\|g\right\|_v,
\end{align*}
hence $\left\|g\right\|_v=v(z_0)\left\|g(z_0)\right\|$, as required.
\end{proof}

We now obtain a similar result for the norm density of $H^\infty_{v\NA}(U,F)$ under the property quasi-$\beta$, a weaker property than the property $\beta$ which was introduced by Acosta, Aguirre and Pay\'a \cite{AcoAguPa-96}.

Note that every Banach space with the property $\beta$ has also the property quasi-$\beta$. For a finite-dimensional Banach space with the property quasi-$\beta$ but not $\beta$, see \cite[Example 5]{AcoAguPa-96}.

\begin{definition}\cite{AcoAguPa-96}\label{def-property-quasi-beta}
A Banach space $F$ has \emph{the property quasi-$\beta$} if there is a subset $A\subseteq S_{F^*}$, a mapping $\sigma\colon A\to S_F$ and a function $\rho\colon A\to\mathbb{R}$ such that the following conditions hold:
\begin{enumerate}
	\item $y^*(\sigma(y^*))=1$ for every $y^*\in A$.
	\item $\left|z^*(\sigma(y^*))\right|\leq\rho(y^*)<1$ for every $y^*,z^*\in A$ with $y^*\neq z^*$.
  \item For every $e^*\in\Ext(B_{F^*})$, there exist a set $A_{e^*}\subseteq A$ and a scalar $t\in\T$ such that $te^*\in\overline{A_{e^*}}^{w^*}$ and $\sup\left\{\rho(y^*)\colon y^*\in A_{e^*}\right\}<1$.
\end{enumerate}
\end{definition}

\begin{theorem}\label{prop-quasi}
Let $U$ be an open subset of a complex Banach space $E$ and let $v$ be a weight on $U$. Suppose that $H^\infty_{v\NA}(U)$ is norm dense in $H^\infty_v(U)$ and let $F$ be a complex Banach space satisfying the property quasi-$\beta$. Then $H^\infty_{v\NA}(U,F)$ is norm dense in $H^\infty_v(U,F)$. 
\end{theorem}

\begin{proof}
We essentially follow the proof of Theorem 2 in \cite{AcoAguPa-96}. Let $\varepsilon>0$ and $f\in S_{H^\infty_v(U,F)}$. By \cite[Proposition 4]{Ziz-73}, % and \cite[Theorem 3.1.11]{Meg-98}, 
there exists $T\in S_{L(G^\infty_v(U),F)}$ such that $\left\|T_f-T\right\|<\varepsilon/2$ and $T^*\in\NA(F^*, G^\infty_v(U)^*)$. By Theorem \ref{main-theo}, $T=T_g$ for some $g\in S_{H^\infty_v(U,F)}$, and we have 
$$
\left\|f-g\right\|_v=\left\|T_f-T_g\right\|=\left\|T_f-T\right\|<\frac{\varepsilon}{2}.
$$
By \cite[Theorem 5.8]{Lim-78}, $T^*$ and, consequently also $g^t=J_v^{-1}\circ (T_g)^*=J_v^{-1}\circ T^*\colon F^*\to H^\infty_v(U)$, attains its norm at a point $e^*\in\Ext(B_{F^*})$, and Definition \ref{def-property-quasi-beta} provides us a set $A_{e^*}\subseteq A$ and a scalar $t\in\T$ such that $te^*\in\overline{A_{e^*}}^{w^*}$ and $r:=\sup\left\{\rho(y^*)\colon y^*\in A_{e^*}\right\}<1$. 

Now, fix a number $0<\gamma<\varepsilon/8$ such that 
$$
1+r\left(\frac{\varepsilon}{4}+2\gamma\right)<\left(1+\frac{\varepsilon}{4}\right)(1-2\gamma)
$$
and since 
$$
1=\left\|g^t\right\|=\left\|g^t(te^*)\right\|_v=\sup\left\{\left\|g^t(y^*)\right\|_v\colon y^*\in A_{e^*}\right\},
$$
we can find $y^*_1\in A_{e^*}$ so that $\left\|g^t(y^*_1)\right\|_v>1-\gamma$. Since $g^t(y^*_1)\in H^\infty_v(U)$, by hypothesis, there exist $g_0\in S_{H^\infty_v(U)}$ and $z_0\in U$ such that $v(z_0)\left|g_0(z_0)\right|=1$ and $\left\|g_0-g^t(y^*_1)/\left\|g^t(y^*_1)\right\|_v\right\|_v<\gamma$. Consequently, we have 
$$
\left\|g_0-g^t(y^*_1)\right\|_v\leq \left\|g_0-\frac{g^t(y^*_1)}{\left\|g^t(y^*_1)\right\|_v}\right\|_v+\left\|\frac{g^t(y^*_1)}{\left\|g^t(y^*_1)\right\|_v}-g^t(y^*_1)\right\|_v
<\gamma+\left(1-\left\|g^t(y^*_i)\right\|_v\right)<2\gamma .
$$
Define $f_0\colon U\to F$ by 
$$
f_0(y)=g(y)+\left[\left(1+\frac{\varepsilon}{4}\right)g_0(y)-g^t(y^*_1)(y)\right]y_1\qquad (y\in U),
$$
where $y_1=\sigma(y^*_1)\in S_F$. Clearly, $f_0\in H^\infty_v(U,F)$ with $\left\|g-f_0\right\|_v<\varepsilon/2$ since 
$$
v(y)\left\|g(y)-f_0(y)\right\|\leq v(y)\frac{\varepsilon}{4}\left|g_0(y)\right|+v(y)\left|g_0(y)-g^t(y^*_1)(y)\right|\leq\frac{\varepsilon}{4}+2\gamma<\frac{\varepsilon}{2}
$$ 
for all $y\in U$. Hence 
$$
\left\|f-f_0\right\|_v\leq \left\|f-g\right\|_v+\left\|g-f_0\right\|_v<\varepsilon.
$$
If we show that $\left\|f_0\right\|_v=v(z_0)\left\|f_0(z_0)\right\|$, the proof will be complete. For it, note first that Condition (iii) in Definition \ref{def-property-quasi-beta} implies that the set $A\subseteq S_{F^*}$ is norming for $F$ %, that is, 
%$$
%\left\|y\right\|=\sup\left\{\left|y^*(y)\right|\colon y^*\in A\right\}\qquad (y\in F)
%$$
and, consequently, 
$$
\left\|(f_0)^t\right\|=\sup\left\{\left\|(f_0)^t(y^*)\right\|_v\colon y^*\in A\right\}.
$$
Since  
$$
(f_0)^t(y^*)=g^t(y^*)+y^*(y_1)\left(\frac{\varepsilon}{4}g_0+g_0-g^t(y^*_1)\right)\qquad (y^*\in F^*),
$$
we have  
$$
\left\|(f_0)^t(y^*)\right\|_v\leq 1+\rho(y^*_1)\left(\frac{\varepsilon}{4}+2\gamma\right)\leq 1+r\left(\frac{\varepsilon}{4}+2\gamma\right)\qquad (y^*\in A,\, y^*\neq y^*_1).
$$
Moreover, $(f_0)^t(y^*_1)=(1+\varepsilon/4)g_0$ and we obtain
$$
\left\|(f_0)^t(y^*_1)\right\|_v=\left(1+\frac{\varepsilon}{4}\right)\left\|g_0\right\|_v>\left(1+\frac{\varepsilon}{4}\right)(1-2\gamma)>1+r\left(\frac{\varepsilon}{4}+2\gamma\right).
$$
Therefore, $\left\|(f_0)^t\right\|=\left\|(f_0)^t(y^*_1)\right\|_v$. It follows that 
\begin{align*}
\left\|f_0\right\|_v&=\left\|(f_0)^t\right\|=\left\|(f_0)^t(y^*_1)\right\|_v=\left(1+\frac{\varepsilon}{4}\right)\left\|g_0\right\|_v=\left(1+\frac{\varepsilon}{4}\right)v(z_0)\left|g_0(z_0)\right|\\
&=v(z_0)\left|(f_0)^t(y^*_1)(z_0)\right|=\left|y^*_1(v(z_0)f_0(z_0))\right|\leq v(z_0)\left\|f_0(z_0)\right\|\leq\left\|f_0\right\|_v,
\end{align*}
and so $\left\|f_0\right\|_v=v(z_0)\left\|f_0(z_0)\right\|$. 
\end{proof}

%%%%%%%%%%%%%%%%%%%%%%%%%%%%%%%%%%%%%%%%%%%%%%%%%%%%%%%%%%%%%%%%%%%%%%%%%%%%%%%%%%%%%%%%%%%%%%%%%%%%%%%%%%%%%%%%%%%%%%%%%%%%%%%%%%%%%%%%%%%%%%%%%%%%%%%%%%%%%%%%%%%%%%%%%%%%%%%%%%%%%%%%%%%%%

\section{The $WH^\infty$-BPB property for mappings with a relatively compact $v$-range}\label{s3}

We now present some versions of the preceding results for mappings $f\in H^\infty_v(U,F)$ such that $vf$ has a relatively compact range in $F$.

\begin{lemma}\label{teo-21}
Let $U$ be an open subset of a complex Banach space $E$ and let $v$ be a weight on $U$. Assume that $\T\At_{G^\infty_v(U)}$ is a norm-closed set of uniformly strongly exposed points of $B_{G^\infty_v(U)}$. Then $H^\infty_{vK\NA}(U,F)$ is norm dense in $H^\infty_{vK}(U,F)$ for every complex Banach space $F$.
\end{lemma}

\begin{proof}
A reading of the proof of \cite[Proposition 1]{Lin-63} shows that for every Banach space $F$, the set
$$
\left\{T\in K( G^\infty_v(U),F)\colon \exists \phi\in \T\At_{G^\infty_v(U)} \, | \, \left\|T(\phi)\right\|=\left\|T\right\|\right\}
$$
is norm dense in $K( G^\infty_v(U),F)$. Let $\varepsilon>0$ and $f\in H^\infty_{vK}(U,F)$. Since $T_f\in K( G^\infty_v(U),F)$ by Theorem \ref{main-theo}, we can find $T\in K( G^\infty_v(U),F)$, $\lambda\in\T$ and $z\in U$ such that $\left\|T(\lambda v(z)\delta_z)\right\|=\left\|T\right\|$ and $\left\|T_f-T\right\|<\varepsilon$. By Theorem \ref{main-theo}, $T=T_{f_0}$ for some $f_0\in H^\infty_{vK}(U,F)$. Similarly, as in the proof of Lemma \ref{prop-1}, we deduce that $f_0\in H^\infty_{v\NA}(U,F)$ and $\left\|f-f_0\right\|_v<\varepsilon$. 
\end{proof}

This result can be improved as follows.

\begin{proposition}\label{teo-22}
Let $U$ be an open subset of a complex Banach space $E$ and let $v$ be a weight on $U$. Assume that $\T\At_{G^\infty_v(U)}$ is a norm-closed set of uniformly strongly exposed points of $B_{G^\infty_v(U)}$. Then $H^\infty_{vK}(U,F)$ has the $WH^\infty$-BPB property for every complex Banach space $F$.
\end{proposition}

\begin{proof}
It is sufficient to note that if $f\in H^\infty_{vK}(U,F)$, then the mappings $g_0,f_0\in H^\infty_v(U,F)$, involved in the proof of Theorem \ref{teo-2}, now belong to $H^\infty_{vK}(U,F)$. Indeed, $g_0$ since $v(g_0-f)$ has a finite dimensional range, and $f_0$ by the application now of Lemma \ref{teo-21}.
\end{proof}

The proof of Theorem \ref{prop-vp} shows that for any $p\geq 1$, $\T\At_{G^\infty_{v_p}(\D)}$ is a norm-closed set of uniformly strongly exposed points of $B_{G^\infty_{v_p}(\D)}$. Therefore,  Proposition \ref{teo-22} yields the following.

\begin{proposition}
$H^\infty_{v_p K}(\D,F)$ with $p\geq 1$ has the $WH^\infty$-BPB property for every complex Banach space $F$.$\hfill\qed$
\end{proposition}

A reading of the proofs of Theorems \ref{prop-beta} and \ref{prop-quasi} show that the following result holds.

\begin{proposition}\label{prop-compacta}
Let $U$ be an open subset of a complex Banach space $E$ and $v$ be a weight on $U$. 
\begin{enumerate}
	\item If $H^\infty_{v}(U)$ has the $WH^\infty$-BPB property, then $H^\infty_{vK}(U,F)$ has the $WH^\infty$-BPB property for any complex Banach space $F$ with the property $\beta$. 
  \item If $H^\infty_{vNA}(U)$ is norm dense in $H^\infty_{v}(U)$, then $H^\infty_{vK\NA}(U,F)$ is norm dense in $H^\infty_{vK}(U,F)$ for any complex Banach space $F$ with the property quasi-$\beta$. $\hfill\qed$
\end{enumerate}
\end{proposition}

Let us recall that a Banach space $F$ is said to be a \emph{Lindenstrauss space} if $F^*$ is isometrically isomorphic to an $L_1(\mu)$ space for some measure $\mu$. 

The next result, influenced by \cite[Theorem 4.2]{AcoBecChoCieKimLeeLoMar-14}, shows that the space of mappings $f\in H(U,F)$ such that $vf$ has a relatively compact range in a Lindenstrauss space $F$ enjoys the $WH^\infty$-BPB property whenever $H^\infty_v(U)$ also has it.

\begin{theorem}
Let $U$ be an open subset of a complex Banach space $E$ and let $v$ be a weight on $U$. Suppose that $H^\infty_v(U)$ has the $WH^\infty$-BPB property. Then $H^\infty_{vK}(U,F)$ has the $WH^\infty$-BPB property for any complex Lindenstrauss space $F$. 
\end{theorem}

\begin{proof}
Let $0<\varepsilon<1$ and let $\varepsilon\mapsto\eta(\varepsilon)$ be the function from $\R^+$ into $\R^+$ which gives the $WH^\infty$-BPB property for $H^\infty_v(U)$. Since $\ell_\infty^n$ has the property $\beta$ for every $n\in\mathbb{N}$ with $\rho=0$, Proposition \ref{prop-compacta} assures that $H^\infty_{vK}(U,\ell_\infty^n)$ enjoys the $WH^\infty$-BPB property with the same function $\varepsilon\mapsto\eta(\varepsilon)$. Note in the proof of Theorem \ref{prop-beta} that if $H^\infty_v(U)$ has the $WH^\infty$-BPB property by a function $\varepsilon\mapsto\eta(\varepsilon)$, then $H^\infty_v(U,F)$ enjoys this property by the same function $\varepsilon\mapsto\eta(\varepsilon)$ which only depends on $U$, but not on the space $Y$ having the property $\beta$. Take  
$$
\eta'(\varepsilon)=\min\left\{\frac{\varepsilon}{4},\eta\left(\frac{\varepsilon}{2}\right)\right\}>0,
$$
and let $f\in S_{H^\infty_{vK}(U,F)}$, $\lambda\in\T$ and $z\in U$ be so that $v(z)\left\|f(z)\right\|>1-\eta'(\varepsilon)$. Choose  
$$
0<\delta<\frac{1}{4}\min\left\{\frac{\varepsilon}{4},v(z)\left\|f(z)\right\|-1+\eta\left(\frac{\varepsilon}{2}\right)\right\}
$$
and let $\{y_1,\ldots,y_n\}$ be a $\delta$-net of $T_f(B_{G^\infty_v(U)})$. By \cite[Theorem 3.1]{LazLin-71}, there exist a natural number $m$ and a subspace $F_0\subseteq F$, isometric to $\ell_\infty^m$, such that $d(y_i,F_0)<\delta$ for all $i\in\{1,\ldots,n\}$. Let $P\colon F\to F_0$ be a surjective projection with $\left\|P\right\|=1$. 

We claim that $\left\|f-P\circ f\right\|_v\leq 4\delta$. Indeed, fix $\phi\in B_{G^\infty_v(U)}$ and so $\left\|T_f(\phi)-y_i\right\|<\delta$ for some $i\in\{1,\ldots,n\}$. Let $y_0\in F_0$ be such that $\left\|y_0-y_i\right\|<\delta$. Then we have 
\begin{align*}
\left\|T_f(\phi)-P(T_f(\phi))\right\|&\leq\left\|T_f(\phi)-y_i\right\|+\left\|y_i-y_0\right\|+\left\|y_0-P(T_f(\phi))\right\|\leq 2\delta+\left\|P(y_0)-P(T_f(\phi))\right\|\\
&\leq 2\delta+\left\|y_0-T_f(\phi)\right\|\leq 2\delta+\left\|y_0-y_i\right\|+\left\|y_i-T_f(\phi)\right\|<4\delta,
\end{align*}
and thus $\left\|T_f-P\circ T_f\right\|\leq 4\delta$. Since $T_f-P\circ T_f\in L(G^\infty_v(U),F)$ and $(T_f-P\circ T_f)\circ\Delta_v=f-P\circ f$, it follows that $T_f-P\circ T_f=T_{f-P\circ f}$ by Theorem \ref{main-theo}. Hence $\left\|f-P\circ f\right\|_v\leq 4\delta$ and this proves our claim. This implies that $\left\|P\circ f\right\|_v\geq \left\|f\right\|_v-4\delta=1-4\delta>0$. Moreover, we have 
$$
v(z)\left\|f(z)-P(f(z))\right\|=\left\|T_f(v(z)\delta_z)-P(T_f(v(z)\delta_z))\right\|<4\delta
$$ 
and therefore,  
$$
v(z)\left\|P(f(z))\right\|>v(z)\left\|f(z)\right\|-4\delta>1-\eta\left(\frac{\varepsilon}{2}\right).
$$
Consequently, $g=(P\circ f)/\left\|P\circ f\right\|_v\colon U\to F_0$ is in $S_{H^\infty_v(U,F)}$ with $v(z)\left\|g(z)\right\|>1-\eta(\varepsilon/2)$. Since $H^\infty_{vK}(U,\ell_\infty^m)$ has the $WH^\infty$-BPB property by the function $\varepsilon\mapsto\eta(\varepsilon)$ and $F_0\subseteq F$ is isometrically isomorphic to $\ell^\infty_m$, there are a mapping $f_0\in H^\infty_{vK}(U,F_0)\subseteq H^\infty_{vK}(U,F)$ with $\left\|f_0\right\|_v=1$, a point $z_0\in U$ and a scalar $\lambda_0\in\T$ so that $v(z_0)\left\|f_0(z_0)\right\|=1$, $\left\|f_0-g\right\|_v<\varepsilon/2$ and $\left\|\lambda v(z)\delta_z-\lambda_0v(z_0)\delta_{z_0}\right\|<\varepsilon/2$. Lastly, we have
$$
\left\|f-f_0\right\|_v\leq\left\|f-P\circ f\right\|_v+\left\|P\circ f-g\right\|_v+\left\|g-f_0\right\|_v<\frac{\varepsilon}{2}+1-\left\|P\circ f\right\|_v+4\delta\leq\frac{\varepsilon}{2}+8\delta<\varepsilon .
$$
\end{proof}

Our next result allows us to transfer the $WH^\infty$-BPB property for mappings in $H^\infty_{vK}$ from range spaces to domain spaces. Its proof is based on \cite[Lemma 3.4]{JohWol-79}.

\begin{proposition}\label{prop-santos}
Let $U$ be an open subset of a complex Banach space $E$, let $v$ be a weight on $U$ and let $F$ be a complex Banach space. Suppose that there exists a net of norm-one projections $(P_i)_{i\in I}\subseteq L(F,F)$ such that $(P_i(y))_{i\in I}$ converges in norm to $y$ for every $y\in F$. If there is a function $\eta\colon\R^+\to\R^+$ such that for every $i\in I$, $H^\infty_{vK}(U,P_i(F))$ has the $WH^\infty$-BPB property by the function $\eta$, then $H^\infty_{vK}(U,F)$ has the $WH^\infty$-BPB property.
\end{proposition}

\begin{proof}
Let $0<\varepsilon<1$ and put 
$$
\eta'(\varepsilon)=\frac{1}{2}\min\left\{\varepsilon,\eta\left(\frac{\varepsilon}{2}\right)\right\}.
$$
Let $f\in S_{H^\infty_{vK}(U,F)}$, $\lambda\in\T$ and $z\in U$ such that $v(z)\left\|f(z)\right\|>1-\eta'(\varepsilon)$. By the relative compactness of $(vf)(U)$ in $F$, we can find a set $\{y_1,\ldots,y_n\}\subseteq F$ such that 
$$
\min\left\{\left\|v(x)f(x)-y_j\right\|\colon 1\leq j\leq n\right\}<\frac{\eta'(\varepsilon)}{3}
$$
for every $x\in U$. By hypothesis, there exists $i\in I$ such that 
$$
\left\|P_i(y_j)-y_j\right\|<\frac{\eta'(\varepsilon)}{3}\qquad (j=1,\ldots,n).
$$
Given $x\in U$, we have 
\begin{align*}
\left\|v(x)(f(x)-Q_{i}(f(x)))\right\|&\leq \left\|v(x)f(x)-y_j\right\|+\left\|y_j-Q_{i}(y_j)\right\|+\left\|Q_{i}(y_j)-Q_{i}(v(x)f(x))\right\|\\
&<2\left\|v(x)f(x)-y_j\right\|+\frac{\eta'(\varepsilon)}{3}
\end{align*}
for all $1\leq j\leq n$, and thus 
$$
\left\|v(x)(f(x)-Q_{i}(f(x)))\right\|\leq 2\min\left\{\left\|v(x)f(x)-y_j\right\|\colon 1\leq j\leq n\right\}+\frac{\eta'(\varepsilon)}{3}<\eta'(\varepsilon).
$$
Therefore, $\left\|Q_{i}\circ f-f\right\|_v\leq\eta'(\varepsilon)$. Clearly, $g=Q_{i}\circ f\in H^\infty_{vK}(U,Q_{i}(F))$ with $\left\|g\right\|_v\leq 1$ and 
$$
v(z)\left\|g(z)\right\|\geq v(z)\left\|f(z)\right\|-\left\|Q_{i}\circ f-f\right\|_v>1-2\eta'(\varepsilon)\geq 1-\eta\left(\frac{\varepsilon}{2}\right).
$$
Since $H^\infty_{vK}(U,P_i(F))$ has the $WH^\infty$-BPB property by the function $\eta$, we can take a mapping $h_0\in S_{H^\infty_{vK}(U,P_i(F))}$, a point $z_0\in U$ and a scalar $\lambda_0\in\T$ such that $v(z_0)\left\|h_0(z_0)\right\|=1$, $\left\|h_0-g\right\|_v<\varepsilon/2$ and $\left\|\lambda v(z)\delta_z-\lambda_0v(z_0)\delta_{z_0}\right\|<\varepsilon/2$. Finally, take $f_0=\iota\circ h_0$, where $\iota$ is the inclusion operator from $P_i(F)$ into $F$. Clearly, $f_0\in S_{H^\infty_{vK}(U,F)}$ with $v(z_0)\left\|f_0(z_0)\right\|=1$ and 
$$
\left\|f_0-f\right\|_v\leq\left\|f_0-g\right\|_v+\left\|g-f\right\|_v=\left\|h_0-g\right\|_v+\left\|g-f\right\|_v<\frac{\varepsilon}{2}+\eta'(\varepsilon)\leq\varepsilon .
$$
\end{proof}

With a similar proof to that of Proposition \ref{prop-santos}, we can state the analogue for the norm density of $H^\infty_{vK\NA}(U,F)$ in $H^\infty_{vK}(U,F)$. 

\begin{proposition}\label{prop-santos-2}
Let $U$ be an open subset of a complex Banach space $E$, $v$ be a weight on $U$ and $F$ be a complex Banach space. Suppose that there exists a net of norm-one projections $(P_i)_{i\in I}\subseteq L(F,F)$ such that $(P_i(y))_{i\in I}$ converges in norm to $y$ for every $y\in F$. If $H^\infty_{vK\NA}(U,P_i(F))$ is norm dense in $H^\infty_{vK}(U,P_i(F))$ for every $i\in I$, then $H^\infty_{vK\NA}(U,F)$ is norm dense in $H^\infty_{vK}(U,F)$. $\hfill\qed$
\end{proposition}

A consequence of Proposition \ref{prop-santos-2} yields the norm density of $H^\infty_{vK\NA}(U,F)$ whenever $F$ is a predual of a complex $L_1(\mu)$-space.

\begin{corollary}
Let $U$ be an open subset of a complex Banach space $E$ and let $v$ be a weight on $U$. Suppose that $H^\infty_{v\NA}(U)$ is norm dense in $H^\infty_{v}(U)$. Then $H^\infty_{vK\NA}(U,F)$ is norm dense in $H^\infty_{vK}(U,F)$ for any complex Lindenstrauss space $F$. $\hfill\qed$
\end{corollary}

\begin{proof}
It suffices to note that every finite subset of a Lindenstrauss space is ``almost'' contained in a subspace of it which is isometrically isomorphic to an $\ell^n_\infty$ space (see \cite[Theorem 3.1]{LazLin-71}). Since all these subspaces are one-complemented and have the property $\beta$, Proposition \ref{prop-compacta} (ii) gives the hypothesis of Proposition \ref{prop-santos-2}.
\end{proof}

\section*{Statements \& Declarations} 

\textbf{Author contributions.} All the authors have the same amount of contribution.

\textbf{Funding.} The first two authors were partially supported by Junta de Andaluc\'{\i}a grant FQM194. The first author was supported by grant PID2021-122126NB-C31 funded by ``ERDF A way of making Europe'' and by MCIN/AEI/10.13039/501100011033. The third author was supported by the Spanish Ministry of Science and Innovation grant TED2021-129189B-C21 and by the aid program for the Spanish University System requalification of the Spanish Ministry of Universities and the European Union-Next GenerationEU.

\textbf{Competing Interests.} The authors have no relevant financial or non-financial interests to disclose.

\textbf{Data Availability.} Data sharing is not applicable to this article as no data sets were generated or analyzed during the current study.

\textbf{Acknowledgements.} The authors are very grateful to Professor Jos\'e Bonet for his helpful comments on some results of this paper.

%%%%%%%%%%%%%%%%%%%%%%%%%%%%%%%%%%%%%%%%%%%%%%%%%%%%%%%%%%%%%%%%%%%%%%%%%%%%%%%%%%%%%%%%%%%%%%%%%%%%%%%%%%%%%%%%%%%%%%%%%%%%%%%%%%%%%%%%%%%%%%%%%%%%%%%%%%%%%

\end{document}